\documentclass[11pt]{article}
\usepackage{enumerate}

\usepackage{amsmath}
\usepackage{amsfonts}
\usepackage{amssymb}
\usepackage[english]{babel}
\usepackage{graphicx,epsf,cite,esint}
\usepackage{amsthm}
\usepackage{mathtools}
\usepackage{textcomp}

\usepackage{hyperref}
\usepackage{accents}
\usepackage[T1]{fontenc}
\usepackage[utf8]{inputenc}
\usepackage{empheq}

\def\XXint#1#2#3{{\setbox0=\hbox{$#1{#2#3}{\int}$ }
\vcenter{\hbox{$#2#3$ }}\kern-.6\wd0}}

\usepackage [autostyle, english = american]{csquotes}
\MakeOuterQuote{"}
\numberwithin{equation}{section}
\theoremstyle{plain}
\newtheorem*{theorem*}{Theorem}
\newtheorem*{lemma*}{Lemma}
\newtheorem{theorem}{Theorem}
\newtheorem{lemma}{Lemma}[section]
\newtheorem{corollary}[lemma]{Corollary}

\theoremstyle{definition}
\newtheorem{definition}[lemma]{Definition}

\def\tr{\operatorname{tr}}

\newcommand{\bbN}{{\mathbb{N}}}
\newcommand{\bbR}{{\mathbb{R}}}

\newcommand{\bbT}{{\mathbb{T}}}

\newcommand{\eps}{\varepsilon}
\newcommand{\del}{\delta}

\newcommand{\beq}{\begin{equation}}
\newcommand{\eeq}{\end{equation}}

\newcommand{\bal}{\begin{align}}
\newcommand{\eal}{\end{align}}
\newcommand{\bals}{\begin{align*}}
\newcommand{\eals}{\end{align*}}
\newcommand{\lb}{\label}

\usepackage{geometry}
\begin{document}
\title{Universal Mixers in All Dimensions}
\author{Tarek M. Elgindi and Andrej Zlato\v{s}}
\footnotetext[1]{Department of Mathematics, UC San Diego. E-mail: telgindi@ucsd.edu.}
\footnotetext[2]{Department of Mathematics, UC San Diego. E-mail: zlatos@ucsd.edu.}
\date{\today}
\maketitle

\begin{abstract}
We construct universal mixers, incompressible flows that mix arbitrarily well general solutions to the corresponding transport equation,  in all dimensions.  This mixing is exponential in time (i.e., essentially optimal) for any initial condition with at least some regularity, and we also show that a uniform mixing rate for all initial conditions cannot be achieved.  The flows are time periodic and uniformly-in-time bounded in spaces $W^{s,p}$ for a range of $(s,p)$ that includes points with $s>1$ and $p>2$.  
\end{abstract}


\section{Introduction}

The problem of mixing via incompressible flows is classical 
and rich with connections to several branches of analysis including PDE, geometric measure theory, ergodic theory, and topological dynamics. 
When diffusion is absent or negligible over the relevant time scales, one can model the process of mixing by the transport equation
\beq \lb{1.1}
\rho_t+u\cdot\nabla \rho = 0,
\eeq
with a fluid velocity $u:Q_d\times\bbR^+\to \bbR^d$ and $Q_d$ being some $d$-dimensional physical domain.  Incompressibility of the advecting fluid requires $u$ to be divergence-free and we also assume  the no-flow boundary condition for $u$, that is, the fluid does not cross the boundary of the domain and satisfies $u\cdot n=0$ on $\partial Q_d\times\bbR^+$.  Since we are interested in the study of mixing in the bulk of the domain and not in effects of  rough boundaries, we will simply assume that $Q_d:=(0,1)^d$ is either the unit cube in $\bbR^d$ or it is $\bbT^d$ (in the latter case opposite sides of $(0,1)^d$ are identified, so $\partial\bbT^d=\emptyset$ and the boundary conditions instead become periodic).

The function $\rho:Q_d\times\bbR^+\to\bbR$ represents the concentration of the mixed quantity with a given initial value $\rho(\cdot,0)$, which we can allow to be negative on account of \eqref{1.1} being invariant with respect to addition of constants.  It will be convenient to take $\rho$ to be mean-zero, so we will always assume that $\int_{Q_d} \rho(x,0)dx=0$.  A special case is when $\rho(\cdot,0)$ is the characteristic function of a subset  of $Q_d$ (minus a constant), modeling the mixing of two fluids.  
 The central question now is how well is it possible to mix a given initial condition via divergence-free flows satisfying some physically relevant constraints, which flows are most efficient at this process, and what is their dependence on the initial condition.


In order to study mixing efficiency of flows, one needs to start with a quantitative definition of how well mixed the advected scalar $\rho$ is at any given time (see the review \cite{T} for various options).  While diffusive mixing results in the time-decay of $L^p$ norms of $\rho$ \cite{CKRZ, DT, STD, TDG} (recall that $\rho$ is mean zero), these norms are all conserved for \eqref{1.1}.  Instead, we  need to use measures of mixing that capture the small-scale variations of $\rho$.  The following two definitions, one geometric in flavor and the other functional-analytic, have been in use recently (see also the discussion at the end of this introduction for a relation to the dynamical systems point of view).
In them, when $Q_d\subseteq\bbR^d$ (i.e., not $\bbT^d$), then we extend the function $f:Q_d\to\bbR$ on the rest of $\bbR^d$ by zero; and as always $\fint_A fdxdy=|A|^{-1}\int_{A} fdxdy$.

\begin{definition} \lb{D.1.1}
Let $f \in L^\infty(Q_d)$ be mean-zero on $Q_d$.
\begin{enumerate}[(i)]
\item   We say that $f$ is \emph{$\kappa$-mixed to scale $\eps$}, with $\kappa,\eps\in(0,1)$, if for each $y\in Q_d$,
\[
\left|\fint_{B_\eps(y)} f(x) dx\right| \leq \kappa \|f\|_{\infty}.
\]
The smallest such $\eps$ is the ($\kappa$-dependent) {\it geometric mixing scale} of $f$.

\item  The {\it functional mixing scale} of $f$ is $\|f\|_{\dot H^{-1/2}}^2\|f\|_{\infty}^{-2}$.
\end{enumerate}
\end{definition}

{\it Remarks.}
1.
The definition in (i) is from \cite{YaoZla}, which was in turn motivated by \cite{Bressan}, where the special case $d=2$, $\kappa= \tfrac 13$, $f(Q_2)=\{-1,1\}$ was considered.  In it, the ``worst-mixed'' region determines the mixing scale, but deviations of size roughly $\kappa$ are tolerated.
\smallskip

2.
The definition in (ii), which is motivated by \cite{MMP} (where $\bbT^d$ was considered instead of $Q_d$ but their analysis easily extend to $Q_d$), does not tolerate deviations but averages the degree of ``mixedness'' of $f$ over all of $Q_d$.  To see the latter, we note that (30) in \cite{MMP} shows for mean-zero functions equivalence of the $\dot H^{-1/2}$-norm and the {\it mix-norm} 
\beq \lb{1.2}
\Phi(f):= \left[  \int_{Q_d\times(0,1)} \left( \fint_{B_r(y)} f(x) dx \right)^2 dy dr \right]^{1/2}.
\eeq  
\smallskip

3.
Other $\dot H^{-s}$-norms of $f$ have been used to quantify mixing, particularly the $\dot H^{-1}$-norm \cite{ACM2,IKX, LLNMD, LTD, S}, with the functional mixing scale being $\|f\|_{\dot H^{-1}}\|f\|_{\infty}^{-1}$ (Wasserstein distance of $f_+$ and $f_-$  has also been used \cite{BOS, OSS, S, S2}).  The latter may sometimes be more convenient than the $\dot H^{-1/2}$-norm and also is directly related to mixing-enhanced diffusion rates when diffusion is present \cite{CKRZ,CZDE}, but it lacks the useful connection to the mix-norm.  We use the $\dot H^{-1/2}$-norm here but note that  our mixing results for it also hold for the $\dot H^{-1}$-norm because the former controls the latter.  The only case when this is not obvious is non-existence of a universal mixing rate in Theorem \ref{T.1.1}(ii), but that proof can be easily adjusted to accommodate the $\dot H^{-1}$-norm (see Remark~2 after Lemma \ref{L.2.3}).
\smallskip

If $\rho$ solves \eqref{1.1} with mean-zero $\rho(\cdot,0) \in L^\infty(Q_d)$ and  divergence free $u:Q_d \times \bbR^+ \to \mathbb{R}^d$, we will say that the flow $\kappa$-mixes  $\rho(\cdot,0)$ to scale $\eps$ by time $t$ whenever $\rho(\cdot, t)$ is $\kappa$-mixed to scale $\eps$. 
In \cite{Bressan,B2}, Bressan conjectured that  if   $\rho(\cdot,0)=\chi_{(0,1/2)\times(0,1)} - \chi_{(1/2,1)\times(0,1)}$ on $Q_2$ is $\tfrac 13$-mixed to some scale $\eps\ll 1$ in time $t$ by a divergence-free $u$, then
\[
\int_0^t \|\nabla u(\cdot, \tau)\|_1 d\tau \ge C |\log\eps|
\]
(with some $\eps$-independent $C<\infty$).  After an appropriate change of the time variable, this is equivalent to
existence of $C<\infty$ such that if a divergence-free $u$ satisfies
\beq \lb{1.3}
\sup_{t>0} \|\nabla u(\cdot, t)\|_1 \le 1
\eeq
and it $\tfrac 13$-mixes $\rho(\cdot, 0)$ to some scale $\eps\ll 1$ in time $t$, then $t\ge C|\log\eps|$.  
That is, given the constraint \eqref{1.3}, the mixing scale cannot decrease super-exponentially in time (in two dimensions).

This {\it rearrangement cost conjecture} of Bressan is still open, but Crippa and De Lellis proved its version with any $\kappa>0$ and
\eqref{1.3} replaced by
\beq\lb{1.4}
\sup_{t>0} \|\nabla u(\cdot,t)\|_p \le 1 
\eeq
with arbitrary $p> 1$  \cite{CL}.
This result also extends to the $\dot H^{-s}$-norm-based definition of mixing via the relationship of the $\dot H^{-1/2}$-norm and the mix-norm \cite{IKX, S}.  

This  motivates the natural question of whether for any mean zero  $\rho(\cdot,0)\in L^\infty(Q_2)$ in two dimensions, there is a flow that satisfies \eqref{1.4} and achieves the (qualitatively optimal for $p>1$) exponential-in-time decay of the mixing scale of the solution $\rho$.  This was recently answered in the affirmative by Yao and Zlato\v s \cite{YaoZla} for any such $\rho(\cdot,0)$ and all $p\in[1,\frac {3+\sqrt 5}2)$, while for $p\ge \frac {3+\sqrt 5}2$ they proved existence of flows achieving mixing scale rates  $O(e^{-t^{\nu_p}/C_p})$ with $\nu_p>\frac 34$ (geometric only).  Alberti, Crippa, and Mazzucato \cite{ACM, ACM2} independently showed that exponential decay can be obtained for all $p\ge 1$, but their result only applies to  special initial data $\rho(\cdot,0)$ (characteristic functions of some regular sets, up to constants). 

In both these works, the mixing flows intricately depend on the initial data, which may not always be very practical.  It is therefore natural to ask whether this dependence can be removed and universally mixing flows exist.  This was done in questions (C) and (D) of the following list of five fundamental questions  from \cite{ACM2} concerning mixing by flows.

\begin{enumerate}[(A)]
\item Given any mean-zero $\rho(\cdot,0)\in L^\infty(Q_2)$, is there a divergence-free $u$ satisfying \eqref{1.4} such that  the mixing scale (functional or geometric) of the solution $\rho$ decays to zero as $t\to\infty$?

\item If the answer to (A) is affirmative, can $u$ be chosen so that the mixing scale decay is exponential?

\item Is there a divergence-free $u$ on $Q_2$ satisfying \eqref{1.4} such that the mixing scale (functional or geometric) of the solution $\rho$ decays to zero as $t\to\infty$ for any mean-zero $\rho(\cdot,0)\in L^\infty(Q_2)$?

\item If the answer to (A) is affirmative, can $u$ be chosen so that the mixing scale decay is exponential?

\item Is there $s>1$, a mean-zero $\rho(\cdot,0)\in L^\infty(Q_2)$, and a divergence-free $u$ satisfying 
\beq\lb{1.5}
\sup_{t>0} \| u(\cdot,t)\|_{W^{s,p}} \le 1 
\eeq
such that the mixing scale (functional or geometric) of the solution $\rho$ decays to zero exponentially as $t\to\infty$?
\end{enumerate}

\noindent 
Additionally, it is natural to state the universal mixing version of (E):

\begin{enumerate}[(F)]
\item If the answer to (E) is affirmative, can a single such $u$ be chosen for all mean-zero $\rho(\cdot,0)\in L^\infty(Q_2)$?
\end{enumerate}

\noindent
The flows from \cite{YaoZla, ACM2} are self-similar in nature with exponentially-in-time decreasing turbulent scales, which appears to be a serious obstacle when it comes to (E) and (F).  Moreover, both \cite{YaoZla, ACM2} only consider the two-dimensional case $d=2$ and their constructions do not appear to easily generalize to higher dimensions, so an obvious seventh question is

\begin{enumerate}[(G)]
\item What are the answers to (A)--(F) in dimensions $d\ge 3$?
\end{enumerate}

The results from \cite{YaoZla} answer (A) and (B) in the affirmative for $p\in[1,\frac {3+\sqrt 5}2)$, as well as (A) in the affirmative for  $p\ge \frac {3+\sqrt 5}2$ and the geometric mixing scale only. (Additionally, \cite{ACM2} proves exponential mixing for all $p\ge 1$ and some special initial data $\rho(\cdot,0)$.)  In this paper we answer (C)--(F) for $p\in[1,\frac {3+\sqrt 5}2)$, as well as the multi-dimensional version (G) of (A)--(F) (except of (B) for very rough $\rho(\cdot,0)$).  For these $p$ and in all dimensions $d\ge 2$, the answers to (C) and (E) are affirmative, while the answers to (D) and (F) are ``no but morally yes''.  More precisely, we construct universally mixing flows in all dimensions that also  yield exponential mixing for all mean-zero initial conditions $\rho(\cdot,0)$ with at least some regularity  (belonging to $H^\sigma(Q_2)$ for some $\sigma>0$, which includes all characteristic functions of regular sets), while we also show that no flow has a fully universal mixing rate (exponential or otherwise).  Moreover, the flows we construct will be periodic in time, with no emergence of small-scale structures as $t\to\infty$.  We also note that all the answers are the same for both the geometric (with any $\kappa>0$) and functional mixing.

The following definition encapsulates our goals.

\begin{definition} \lb{D.1.2}
\begin{enumerate}[(i)]
\item   A divergence-free $u:Q_d\times\bbR^+\to\bbR^d$ is called a {\it universal mixer} (in the functional sense or geometric sense for some fixed $\kappa\in(0,1)$) if for any $\rho(\cdot,0)\in L^\infty(Q_d)$ the mixing scale (functional or geometric with the given $\kappa$) of the solution to \eqref{1.1} converges to 0 as $t\to\infty$.  

\item  A universal mixer $u$ (in either sense from (i)) has  {\it mixing rate} $\lambda:\bbR^+\to\bbR^+$ with $\lim_{t\to\infty} \lambda(t)=0$ if for each $\rho(\cdot,0)\in L^\infty(Q_d)$ 
there is $\tau<\infty$ such that the mixing scale of $\rho(\cdot,t)$ (in the relevant sense)  is at most $\lambda(t-\tau)$ for all $t>\tau$.

\item  A universal mixer $u$ (in either sense from (i)) is a {\it universal exponential  mixer}  if there is $\gamma>0$ such that $u$ has mixing rate $\lambda(t)=e^{-\gamma t}$.

\item  A universal mixer $u$ (in either sense from (i)) is an {\it almost-universal exponential mixer}  if for each $\sigma>0$ there is $\gamma_\sigma>0$ such that for each $\rho(\cdot,0)\in L^\infty(Q_d)\cap H^\sigma(Q_d)$ 
 there is $\tau<\infty$ such that the mixing scale of $\rho(\cdot,t)$ (in the relevant sense) is at most $e^{-\gamma_\sigma(t-\tau)}$ for all $t>\tau$.

\item  If $u$ is a universal mixer in the geometric sense for each $\kappa\in(0,1)$, we say that $u$ is a {\it universal mixer in the geometric sense}.
If $u$ is an (almost-)universal exponential mixer in the geometric sense for each $\kappa\in(0,1)$, with the relevant exponential mixing rates independent of $\kappa$, then we say that $u$ is an {\it (almost-)universal exponential mixer in the geometric sense}.
\end{enumerate}
\end{definition}

Here are our main results, for $Q_d$ being either the unit cube or the torus $\bbT^d$.

\begin{theorem}\label{T.1.1}
\begin{enumerate}[(i)]
\item
For any $d\ge 2$, there is a divergence-free time-periodic vector field $u$ on $Q_d$ satisfying no-flow (or periodic) boundary conditions such that $u\in L^\infty([0,\infty); W^{s,p}(Q_d))$ (or $u\in L^\infty([0,\infty); W^{s,p}(\bbT^d))$) for any $s<\frac{1+\sqrt 5}2$ and $p\in[1,\frac{2}{2s+1-\sqrt{5}})$, and $u$ is a universal mixer and an almost-universal exponential mixer in both the geometric and functional sense.

\item For any $d\ge 2$, there is no divergence-free universal mixer on $Q_d$ satisfying either no-flow or periodic boundary conditions that has a mixing rate, in either the functional sense or the geometric sense with some $\kappa\in(0,1)$.  In particular, there are no universal exponential mixers in any dimension.
\end{enumerate}
%
\end{theorem} 

{\it Remarks.}
1.  Hence $(s,p)$ with $s>1$ and $p>2$ are included in (i).  As mentioned above, prior mixing results required two dimensions $d=2$, as well as $u$ that depended on the mixed function and only belonged to spaces $L^\infty([0,\infty); W^{1,p}(Q_2))$.  
\smallskip

2.  The one-period flow map of our flow in (i) on $Q_2$ with no-flow boundary conditions is the folded Baker's map, and the other cases are also related to it (see Section \ref{BakersMap}).  Thus it is exponentially mixing in the dynamical systems sense, and topologically conjugate to the Smale Horseshoe map.\smallskip

3. Our $u$ in (i) will in fact only take finitely many (two when $d=2$) distinct values as a function of time.  This is virtually the simplest possible time-dependence, and we do not know whether time-independent universal mixers exist on $Q_d$ (they do not for $d=2$, due to divergence-free vector fields on $Q_2$ having a stream function).
\smallskip

4.  While the flows we construct in (i) are discontinuous in time, they can be made smooth in time by a simple re-parametrization described in \cite{YaoZla}.  In space, these flows are H\" older continuous, and smooth away from a finite number of hyperplanes.
\smallskip

5. (i) also shows that replacing the $W^{1,p}$-norm of $u$ in the constraint \eqref{1.4} by the $W^{s,p}$-norm for $(s,p)$ as in (i) does not qualitatively improve the exponential upper bound on mixing efficiency of flows (i.e., exponential lower bound on the mixing scale) from \cite{CL,IKX,S}.
\smallskip

6. We also remark that it follows from (i) that mean-zero solutions to the corresponding advection-diffusion equation 
\[
\rho_t+u\cdot\nabla\rho=\nu\Delta\rho
\]
 with $\nu\in(0,\frac{1}{2})$ satisfy 
 \[
 \|\rho(\cdot,t)\|_{2}\leq e^{-c\nu^{0.62}t}\|\rho(\cdot,0)\|_{2}
 \] 
 for all $t>\nu^{-0.62}$, with $c$ a universal constant (see Theorem 4.4 in \cite{CZDE}).  That is, their diffusive time scale is at most $O(\nu^{-0.62})$, which is much shorter than the time scale $O(\nu^{-1})$ for the heat equation (without advection) when $0<\nu\ll 1$.  However, this estimate is still not optimal.
\smallskip

%

We do not know whether universal mixers or almost universal exponential mixers with a higher regularity than ours exist.  One possible candidate for an efficient universal mixer in two dimensions (on $\bbT^2$) is almost every realization of the random vector field taking values $(\sin(2\pi x_2+\omega_n),0)$ and $(0,b\sin(2\pi x_1+\omega_n))$ on time intervals $(n-1,n-\frac 12]$ and $(n-\frac 12, n]$, respectively, for $n\in \bbN$ and with $\omega_n$ independent random variables uniformly distributed over $\bbT$.
While these alternatively horizontal and vertical shear flows, considered by Pierrehumbert in \cite{Pie}, appear to have very good mixing properties, we are not aware of their rigorous proofs.

\bigskip
{\bf Discussion of related dynamical systems results.} 
The study of mixing maps and flows has a rich history. While there are a plethora of \emph{maps} that are known to be good mixers, there are only a few examples of \emph{flows} that mix well (see \cite{CFS}). One important class of examples are Anosov flows, introduced in \cite{Anosov}. A flow $\Phi_t:M\rightarrow M$ on a  compact Riemannian manifold $M$ is an Anosov flow if at each $x\in M$, the tangent space $TM_x$ can be decomposed into three subspaces, one contracting, one expanding, and one that is 1-dimensional and corresponds to the direction of the flow.  It was shown in important works of Dolgopyat \cite{Dolgopyat} and Liverani \cite{Liverani} that all smooth enough Anosov flows are exponentially mixing in the sense of the decay of correlations (which implies exponential mixing in the sense of Definition \ref{D.1.1}). Anosov flows are known to exist in a number of settings, the most concrete of which seems to be as geodesic flows on certain negatively curved Riemannian manifolds of dimensions $d\ge 3$. 
A very interesting open problem is to construct an incompressible velocity field in the flat geometry of $\mathbb{T}^d$ for $d\geq 3$ that is smooth uniformly in time, and whose flow is an exponential mixer.  Theorem \ref{T.1.1}(i) provides a H\" older continuous time-periodic example on both $\bbT^d$ and $(0,1)^d$ for $d\ge 2$.
 (We note  that A. Katok constructed mixing flows on all smooth two dimensional manifolds \cite{K}, but their mixing rates are only logarithmic or algebraic in time, depending on the allowed regularity of the flows \cite{Dolgopyat2}.) 

It is also easy to show that many linked twist maps lead to exponential mixers on  $\mathbb{T}^2$ (see \cite{SpringhamThesis}). For example, Arnold's cat map can be realized as the composition of maps $(x,y)\mapsto (x+y,y)$ and $(x,y)\mapsto  (x,x+y)$, which are both flow maps of divergence-free velocity fields (exponential mixing in this case is not difficult to establish).  However, these velocity fields are discontinuous on $\mathbb{T}^2$, and although they belong to the space $BV,$ they do not belong to $W^{s,p}$ for any $s,p\geq 1$. To the best of our knowledge, prior to this work  there were no known exponential mixers on $\mathbb{T}^d$ or $(0,1)^d$ (for any $d\geq 2$) with regularity higher than BV (see \cite{DKK} for even more irregular examples).  Allowing the velocity fields to be discontinuous also trivially allows for flow maps with rigid and "cut-and-paste" motions, which can be easily combined to produce exponential mixers.  The requirement of continuity, let alone the constraint \eqref{1.5} when $s,p>1$, adds non-trivial difficulties.   

It is also important to emphasize that none of the above examples, whether on negatively curved manifolds or even on $\bbT^d$, answers questions about mixing in physically relevant real world settings, that is, bounded domains in $\bbR^2$ and $\bbR^3$.  Theorem \ref{T.1.1}(i) and the (initial-data-dependent) two-dimensional flows from \cite{YaoZla} are the only exponentially mixing flows in this setting that we are currently aware of.

\bigskip
{\bf Acknowledgements.}  AZ acknowledges partial support by NSF grants DMS-1652284 and DMS-1656269. TME acknowledges partial support by NSF grant DMS-1817134. We also thank Dmitri Dolgopyat for pointing us to \cite{K,DKK}, Gautam Iyer for mentioning to us  \cite{Pie}, and Amir Mohammadi and  Sheldon Newhouse for useful discussions.

%
%
%

\section{Discrete Time Mixing}\label{BakersMap}

Since our flow will be time-periodic, a crucial step will be the analysis of its flow map at integer multiples of its period.  Of course, this is just the sequence of powers of the flow map for a single period.  We note that Definition \ref{D.1.2} naturally extends to the case of the initial datum $\rho(\cdot,0)$ and the solution $\rho$ being replaced by, respectively, $f\in L^\infty(Q_d)$ and the sequence of functions $\{f\circ T^{-n}\}_{n\ge 1}$ obtained by repeatedly applying some measure-preserving bijection $T:Q_d\to Q_d$ to $f$.  Or, more generally, with  $\{f\circ T_{n}^{-1}\}_{n\ge 1}$ instead of $\{f\circ T^{-n}\}_{n\ge 1}$, where $T_n:Q_d\to Q_d$ is some sequence of measure-preserving bijections.  In those cases we will call $T$ resp. $\{T_{n}\}_{n\ge 1}$ (almost-)universal (exponential) mixers on $Q_d$.

 We start with the two-dimensional case, in which we will use the notation $(x,y)\in Q_2$, and afterwards extend our analysis to all dimensions.  We note that the maps considered here will only work as flow maps in the no-flow case $Q_d=(0,1)^d$.  Adjustments needed for the periodic case $\bbT^d$ will be performed in the next section.

\subsection{Construction of a discrete time universal mixer $T:Q_2\to Q_2$}

There are various maps on the square $Q_2=(0,1)^2$ that are known to have good mixing properties,  one such example being {\it Baker's map}.  The classical Baker's map $B:Q_2\to Q_2$ is obtained by cutting the square vertically in two halves and mapping these affinely onto the upper and lower halves of $Q_2$, with no rotation.  Unfortunately, it seems that divergence-free velocity fields whose flow map at some time is $B$ are no more regular than BV.    It turns out, however, that a closely related map, for which one of the rectangles is rotated by $180^\circ$ (the {\it folded Baker's map} in Figure \ref{fig1} below), has very similar mixing properties and can be realized via incompressible flows with higher regularity.  We analyze it now, and will construct the relevant velocity field in the next section.

\begin{definition} \label{D.2.1}
Let $Q_2:=(0,1)^2$, and 
\[
Q_2':=Q_2\setminus \{(x,y)\in Q_2\,|\, \text{$2^k x\in\mathbb Z$ or $2^k y\in\mathbb Z$ for some $k\in\mathbb N$}\}.
\]
Define $T: Q_2'\rightarrow Q_2'$ by 
\beq\lb{2.0}
T(x,y)=\begin{cases} 
     - (2x, \frac{y}{2}) + (1,\frac{1}{2}) & x\in(0,\frac{1}{2}), \\
     (2x, \frac{y}{2}) + (-1,\frac{1}{2}) & x\in (\frac{1}{2},1).   \end{cases}
\eeq
For $k\in\mathbb{N}\cup\{0\}$ and $j\in \mathbb{Z}\cap[0,2^{k})$, let
$$H^j_k:= \left[ \left(0,1 \right)\times \left(\frac j{2^{k}}, \frac {j+1}{2^{k}} \right) \right] \cap Q_2',$$
$$V^j_k:= \left[ \left(\frac j{2^{k}}, \frac {j+1}{2^{k}} \right)\times \left(0,1 \right) \right] \cap Q_2'$$
be the {\it horizontal and vertical dyadic strips of width $2^{-k}$}, respectively.   Finally, let 
\[
G_{k,k'}:= \left\{H^j_k\cap V^{j'}_{k'} \,\Big|\, j\in \mathbb{Z}\cap[0,2^{k}) \text{ and } j'\in \mathbb{Z}\cap[0,2^{k'}) \right\}
\]
 be the collection of all  {\it dyadic rectangles} of size $2^{-k'}\times 2^{-k}$ and let $G_k:=G_{k,k}$  be the collection of all  {\it dyadic squares} of size $2^{-k}\times 2^{-k}$.
\end{definition}

\begin{figure}[htbp]
\centerline{\includegraphics[scale=0.6]{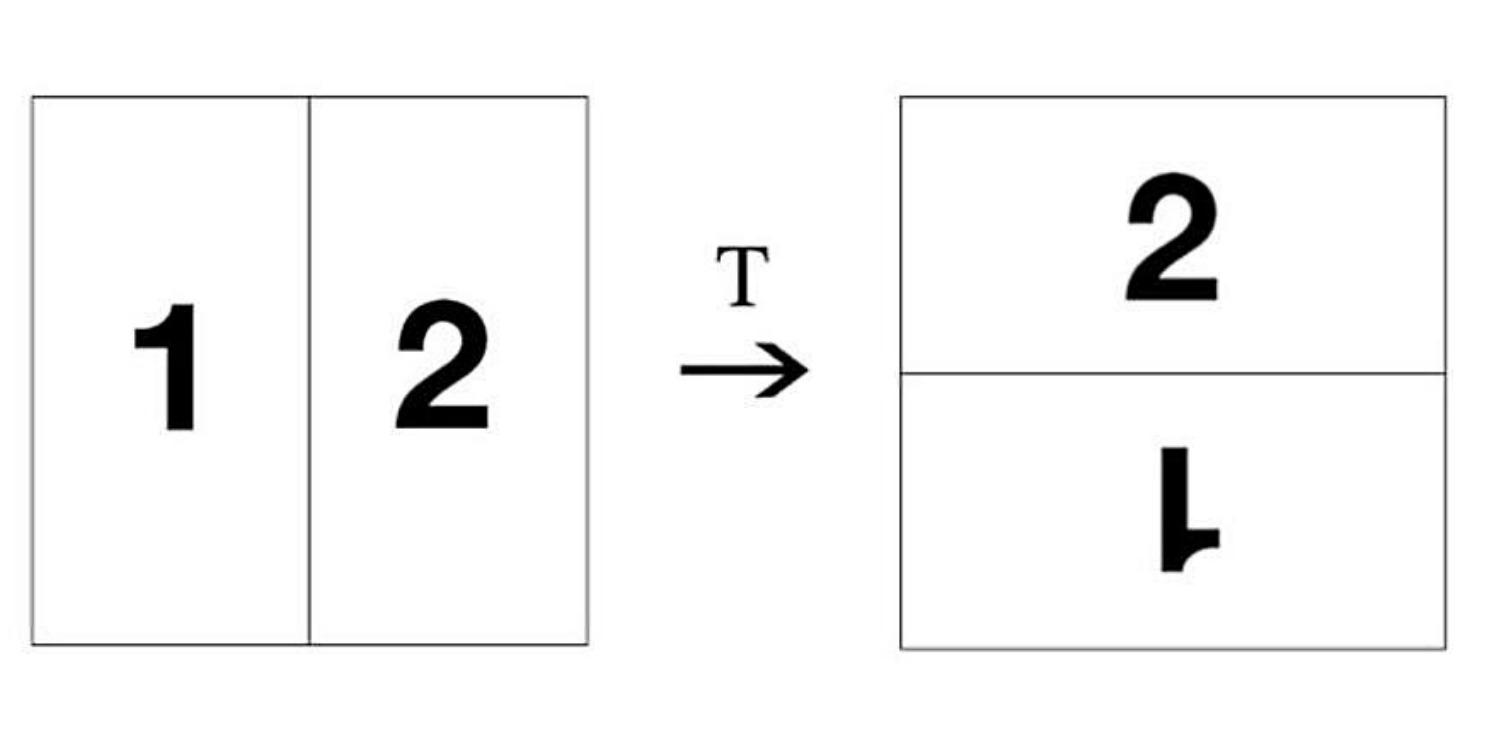}}
\caption{Folded Baker's map $T$.}\label{fig1}
\end{figure}

%


The restriction of the above definitions to $Q_2'$ is due to the fact that the map $T$ is not defined on the line $\{x=\frac 12\}$ and more generally, $T^l$ is not defined on $Q_2\cap \{2^l x\in\mathbb Z\}$ for any $l\in\mathbb N$.  Similarly, $Q_2\cap \{2^l y\in\mathbb Z\}$ is not in the image of $T^l$.
However, $T$ is a bijection on the full-measure set $Q_2'$, and restriction to $Q_2'$ will  avoid some technicalities.  


%
%

\begin{lemma}\label{ImageOfHorizontal}
Let  $k,k',l\in\mathbb{N}$, and let $j\in \mathbb{Z}\cap[0, 2^{k})$ and $j'\in \mathbb{Z}\cap[0, 2^{k'})$.

\noindent (i)
If $i\in \mathbb{Z}\cap[0, 2^{l})$, then $T^l(H^j_k)\cap H^i_l$ is a single horizontal dyadic strip of width $2^{-(k+l)}$.  

\noindent (ii)
If $l\le k'$, then $T^{l}(H^j_k\cap V^{j'}_{k'})$ is a single horizontal dyadic rectangle from $G_{k+l,k'-l}$.
\end{lemma}

{\it Remarks.} 1.  That is, any dyadic rectangle $H^j_k\cap V^{j'}_{k'}$ is being doubled horizontally and halved vertically by repeated applications of $T$ until it becomes a horizontal dyadic strip.  Continued applications of $T$ then always double the number of horizontal dyadic strips and halve their widths, while these strips become fairly regularly distributed throughout $Q_2'$.  One can  write down recursive formulas for $T^l(H^j_k)$ and $T^{l}(H^j_k\cap V^{j'}_{k'})$, but we will not need these.
\smallskip

2.  Baker's map $B$, which coincides with $T$ for $x\in(\frac 12,1)$ but $B(x,y)= (2x, \frac{y}{2})$ for $x\in(0,\frac{1}{2})$, has the same properties and the distances between the strips constituting $B^l(H^j_k)$ are the same.

\begin{proof}
Both statements are immediate by induction on $l$, using that the bijections $T:V^{0}_1\to H^{0}_1$ and $T:V^{1}_1\to H^{1}_1$ are both of the form $(x,y)\mapsto\pm(2x,\frac y2)+(\mp1,\frac 12)$.
\end{proof}

This now directly implies that $T$ is mixing in the classical sense.  For the sake of completeness, let us extend $T$ (bijectively) to all of $Q_2$ by \eqref{2.0} for $(x,y)\in Q_2\setminus[\{\frac 12\}\times(0,1)]$ and by $T(x,y):=(y,x)$ for  $(x,y)\in\{\frac 12\}\times(0,1)$.

\begin{lemma}\label{UniversalDiscrete}
If $A, B\subseteq Q_2$ are measurable, then $$\lim_{n\rightarrow\infty} |T^n(A)\cap B|= |A|\cdot |B|.$$
Similarly, if $f,g\in L^2(Q_2)$, then 
\[
\lim_{n\to\infty} \int_{Q_2}(f\circ T^{-n}) \, g\,dxdy = \int_{Q_2} f\, dxdy \int_{Q_2} g\,dxdy.
\]
\end{lemma}

\begin{proof}
Assume first that $A$ and $B$ are each a disjoint union of dyadic squares $H^j_k\cap V^{j'}_{k}$ (with a fixed $k$). Applying first Lemma \ref{ImageOfHorizontal}(ii) with $l:=k$ and then Lemma \ref{ImageOfHorizontal}(i) with $l:=n-k$, we see that for $n\ge 2k$ we have
\begin{equation} \label{2.1}
|T^{n}(H^j_k\cap V^{j'}_{k}) \cap (H^{i}_k\cap V^{i'}_{k})| = 2^{-4k} = |H^j_k\cap V^{j'}_{k}|\cdot |H^{i}_k\cap V^{i'}_{k}|
\end{equation}
for any $j,j',i,i'\in \mathbb{Z}\cap[0, 2^{k})$ because the intersection on the left is a single dyadic rectangle of size $2^{-k} \times 2^{-3k}$.
Hence $|T^n(A)\cap B|= |A|\cdot |B|$ for all $n\ge 2k$.  For general $A$ and $B$, the first claim follows  via  approximation by disjoint unions of  dyadic squares.

The second claim follows from the first via approximation of $f$ and $g$ by simple functions.
\end{proof}


%

This shows that $T$ is a (discrete time) universal mixer.

\begin{lemma} \label{L.2.2}
$T$ is a universal mixer on $Q_2$ in both the geometric and functional senses.
\end{lemma}

\begin{proof}
From the previous result we know that $\{f\circ T^{-n}\}_{n\ge 1}$ converges to $0$ weakly in $L^2(Q_2)$ for any mean-zero $f\in L^\infty(Q_2)$, and hence strongly in $\dot H^{-1/2}(Q_2)$.  This shows that $T$ is a universal mixer in the functional sense, and the geometric sense claim follows from this and Lemma A.1 in \cite{YaoZla}.
\end{proof}

\subsection{Non-existence of a uniform mixing rate in any dimension}

We will now show that Lemma \ref{L.2.2} is sharp in the sense that no sequence of measure preserving bijections on $Q_d$ (for any $d\ge 2$) with no-flow or periodic boundary conditions possesses a  universal (functional or geometric) mixing rate.

\begin{lemma}\label{L.2.3}
Let $\{T_n\}_{n\ge 1}$ be a sequence of measure-preserving bijections $T_n:Q_d\to Q_d$ and let $\lambda_n>0$ satisfy $\lim_{n\to\infty} \lambda_n=0$.  Then there is a measurable set $A\subseteq Q_d$ with $|A|=\frac 12$ and a sequence $n_j\to\infty$ as $j\to\infty$ such that $T_{n_j}(A)$ contains a ball with radius $\lambda_{n_j}^{1/d}$ for each $j\in\mathbb N$.  In particular, if $f:=\chi_A - \chi_{Q_d\setminus A}$, then $\int_{Q_d} f \,dxdy=0$ but $f\circ T_{n_j}^{-1}$ is not $\kappa$-mixed to scale $\lambda_{n_j}$ for any $\kappa\in (0,1)$ and $j\in\mathbb N$, as well as  $\|f\circ T_{n_j}^{-1}\|_{\dot H^{-1/2}(Q_d)}> \lambda_{n_j}$ for all large enough $j$.
\end{lemma}

{\it Remarks.}  1.  One can also show that for any measurable $ B\subseteq Q_d$, there is a measurable $ A\subseteq Q_d$ such that
\[
 \limsup_{n\rightarrow\infty} \frac{\big| |T_n(A)\cap B| - |A|\cdot |B|\big|}{\lambda_n}=\infty.
 \]
  \smallskip

2.  While the last claim does not imply $\|f\circ T_{n_j}^{-1}\|_{\dot H^{-1}(Q_d)}> \lambda_{n_j}$, the construction from the proof below can be easily adjusted to achieve this (using the definition of the $\dot H^{-1}$-norm).

\begin{proof}
Pick $n_j$ so that $V_d\sum_{j=1}^\infty \lambda_{n_j}\le \frac 12$, with $V_d$ the volume of the unit ball in $\bbR^d$, and let $A\subseteq Q_d$ be any measurable set with $|A|=\frac 12$ containing $\bigcup_{j\ge 1}T^{-1}_{n_j}(B_{\lambda_{n_j}^{1/d}}(\frac 12,\dots,\frac 12))$.  This proves the first two claims. The last claim follows from this and equivalence of the $\dot H^{-1/2}$-norm and the mix-norm \eqref{1.2}, because $f\circ T_{n_j}^{-1}\equiv 1$ on at least $O(\lambda_{n_j})$ proportion of all the balls in $Q_d$ with any fixed radius $r\in(0,\frac 12 \lambda_{n_j}^{1/d})$ (hence the mix norm of $f\circ T_{n_j}^{-1}$ is at least $O(\lambda_{n_j}^{(d+1)/2d})$).
\end{proof}

\begin{lemma}\label{L.2.4}
For any $d\ge 2$, there is no universal mixer $\{T_n\}_{n\ge 1}$ with measure-preserving bijections $T_n:Q_d\to Q_d$ that has some mixing rate, in either the functional sense or the geometric sense with some $\kappa\in(0,1)$.  
\end{lemma}

\begin{proof}
Given any rate $\lambda$ with $\lim_{n\to\infty} \lambda(n)=0$, pick a sequence $\{\lambda'_n\}_{n\ge 1}$ of positive numbers that decays more slowly than $\{\lambda(n+m)\}_{n\ge 1}$ for any $m\in\bbN$.  Then apply Lemma \ref{L.2.3}.
\end{proof}

\subsection{Almost-universal exponential mixing by $T:Q_2\to Q_2$}

Finally, part (iii) of the following lemma shows that $T$ is an almost-universal exponential mixer in both the geometric and functional senses. 
%
%

\begin{lemma} \label{L.2.5}
Assume that $f\in L^\infty(Q_2)$ with $\int_{Q_2} f \,dxdy=0$.  For $A\subseteq Q_2$, let $\mathcal C_r(A)$ be the set of $(x,y),(x',y')\in A$ such that $|(x,y)-(x',y')|<r$ and the closed straight segment connecting $(x,y)$ and $(x',y')$ is contained in $A$.
\begin{enumerate}[(i)]
\item
If there is 
a set $\Gamma\subseteq Q_2$ with box-counting dimension $\gamma<2$
and $r>0$ such that 
\[
\sup_{(x,y),(x',y')\in \mathcal C_r(Q_2\setminus\Gamma)} |f(x,y)-f(x',y')| < \kappa \|f\|_\infty,
\]
then there is $m\in\bbN$ such that  $f\circ T^{-n}$ is $\kappa$-mixed  to scale $2^{-(n-m)/2}$ for each $n\in\bbN$.
\item 
If there is a set $\Gamma\subseteq Q_2$ with box-counting dimension $\gamma<2$
and $\alpha>0$ such that 
\[
\sup_{(x,y),(x',y')\in \mathcal C_1(Q_2\setminus\Gamma)} \frac{|f(x,y)-f(x',y')|} {|(x,y)-(x',y')|^\alpha} < \infty,
\]
then the functional mixing scale of $f\circ T^{-n}$ decreases exponentially at a rate that only depends on $\min\{\alpha, 2-\gamma\}$.
\item
If  $f\in H^\sigma(Q_2)$ for some $\sigma\in(0,1]$, then the functional mixing scale  of $f\circ T^{-n}$ and its geometric mixing scale for any $\kappa\in(0,1)$ decrease exponentially at a rate that only depends on $\sigma$.
\end{enumerate}
\end{lemma}

{\it Remark.}  The hypothesis in (ii) means that  $f$ is H\" older continuous on $Q_2$ except possibly ``across'' $\Gamma$.  

\begin{proof}
(i) 
Let the supremum above be $\kappa' \|f\|_\infty$, with $\kappa'<\kappa$ and let $\gamma'\in(\gamma, 2)$.  Fix any $k\in\mathbb N$ so that $2^{-k}<\frac r2$.  For $j,j'\in\mathbb Z\cap [0,2^{k})$ let $a_{j,j'}$ be the average of $f$ over the dyadic square $H_k^j\cap V_k^{j'}$,  and let $g(x,y):=a_{j,j'}$ for $(x,y)\in H_k^j\cap V_k^{j'}$.  Note that $\sum_{j,j'=0}^{2^{k}-1} a_{j,j'}=0$ because $\int_{Q_2} f \,dxdy=0$.

Just as in the proof of Lemma \ref{UniversalDiscrete}, we find that the intersection of any dyadic square $Q\in G_k$ with $T^{2k}(Q')$ for any dyadic square $Q'\in G_k$ is a single dyadic rectangle of size $2^{-k}\times 2^{-3k}$ (and the latter are disjoint for distinct $Q'\in G_k$).  Hence 
\[
\left| \fint_{Q} f\circ T^{-2k}\,dxdy \right| \le \left| \fint_{Q} g\circ T^{-2k}\,dxdy \right|  + \left| \fint_{Q} (f-g)\circ T^{-2k}\,dxdy \right|  \le  0 + \kappa' \|f\|_\infty + 2^{1+(\gamma'-2)k}  \|f\|_\infty
\]
because $\sum_{j,j'=0}^{2^{k}-1} a_{j,j'}=0$,  and $|f-g|$ is bounded by $2\|f\|_\infty$ and exceeds $\kappa' \|f\|_\infty$ on fewer than $2^{\gamma'k}$ squares $Q'\in G_k$ when $k$ is large enough (namely those whose closures contain points from $\Gamma$; then the intersection of $Q$ with $T^{2k}$ applied to these ``bad'' squares has measure less than $2^{(\gamma'-4)k} $).
Since $\kappa'<\kappa$ and $\gamma'<2$ we find that
\[
\left| \fint_{Q} f\circ T^{-2k}\,dxdy \right| \le \frac{ \kappa'+\kappa}2 \|f\|_\infty 
\]
for all large $k$ and any $Q\in G_k$ (a similar analysis shows the same for $T^{-(2k+1)}$ in place of $T^{-2k}$). The claim now  follows easily by taking a  large enough $m$ (depending on $\kappa-\kappa'$) and $k:=\lfloor\frac n2\rfloor$ for each $n\in\mathbb N$, so that the absolute value of the average of  $f\circ T^{-n}$ over any ball of radius $2^{-(n-m)/2}$ cannot exceed the maximum of the absolute values of its averages over all the squares from $G_k$ by more than $\frac{ \kappa-\kappa'}2 \|f\|_\infty$.

(ii)  
This is similar to (i), but with $\kappa' \|f\|_\infty$ replaced by $C2^{-\alpha (k-1)}$ due to the hypothesis (with $C$ the supremum in the statement of (ii)).  Hence we get
\[
\left| \fint_{Q} f\circ T^{-2k}\,dxdy \right| \le 2^{1-\min\{\alpha, 2-\gamma\}k}  \left( C + \|f\|_\infty \right)
\]
for all large $k$ and any $Q\in G_k$ (and again the same holds with $T^{-(2k+1)}$ in place of $T^{-2k}$).  The claim now easily follows from the equivalence of the $\dot H^{-1/2}$-norm and the mix-norm \eqref{1.2}.

(iii)
Let $f=\sum_{n\in\mathbb N^2} a_{n}\phi_{n}$ with $\phi_{n}$ trigonometric polynomials such that $\|\phi_n\|_\infty\le 1$ and $\|\nabla\phi_n\|_\infty\le C|n|$, as well as 
\[
C^{-1} \sum_{n\in\mathbb N^2} |a_{n}|^2|n|^{2r} \le 
\|f\|_{W^{r,2}}^2\le C \sum_{n\in\mathbb N^2} |a_{n}|^2|n|^{2r}
\]
 for some  $C\ge 1$ and $r\in\{0,\sigma\}$.  If $f_k:=\sum_{|n|\le 2^{k/3}} a_{n}\phi_{n}$, then the argument in (i) shows for any square $Q\in G_k$,
 \beq\lb{2.3}
\left| \fint_{Q} f_k\circ T^{-2k}\,dxdy \right| \le 2^{1-k}   \sum_{|n|\le 2^{k/3}} |a_{n}| C|n| \le C 2^{1-k} \left (\sum_{|n|\le 2^{k/3}} |a_{n}|^2 \right)^{1/2} 2^{2k/3} \le C^2 2^{1-k/3} \|f\|_{H^\sigma}.
\eeq
Equivalence of the $\dot H^{-1/2}$-norm and the mix-norm \eqref{1.2}, together with
 \beq\lb{2.4}
\|(f-f_k)\circ T^{-2k}\|_{\dot H^{-1/2}}\le \|(f-f_k)\circ T^{-2k}\|_{L^2} = \|f-f_k\|_{L^2} \le C2^{-k\sigma/3} \|f\|_{H^\sigma}
\eeq
and the fact that both claims again hold with $T^{-(2k+1)}$ in place of $T^{-2k}$, now yields the functional mixing scale claim.  The geometric mixing scale claim follows from this and Lemma A.1 in \cite{YaoZla}.
\end{proof}

\subsection{Universal mixing and almost-universal exponential mixing on $Q_d$}

Let us now consider the case $x=(x_1,\dots,x_d)\in Q_d$ with $d\ge 3$ instead of $(x,y)\in Q_2$.

\begin{definition} \label{D.2.7}
Let $Q_d:=(0,1)^d$, and 
\[
Q_d':=Q_d\setminus \{x\in Q_d\,|\, \text{$2^k x_i\in\mathbb Z$ for some $i\in\{1,\dots,d\}$ and $k\in\mathbb N$}\}.
\]
For $i\in\{1,\dots,d-1\}$, let $T_{d,i}:Q_d\to Q_d$ be given by 
\[
T_{d,i}(x):= (x_1,\dots,x_{i-1},[T(x_i,x_d)]_1, x_{i+1},\dots,x_{d-1},[T(x_i,x_d)]_2),
\]
where $T$ is from Definition \ref{D.2.1} and $[T]_j$ is its $j^{\rm th}$ coordinate (with $T$ extended to all of $Q_2$ as before Lemma \ref{UniversalDiscrete}).   Also let $T_{d}:Q_d\to Q_d$ be given by
\[
T_d:= T_{d,d-1}\circ\dots\circ T_{d,1}.
\]
For $k\in\mathbb{N}\cup\{0\}$ and $j_1,\dots,j_{d}\in \mathbb{Z}\cap[0,2^{k})$, let
$$H^{j_d}_k:= \left[ \left(0,1 \right)^{d-1}\times \left(\frac {j_d}{2^{k}}, \frac {j_d+1}{2^{k}} \right) \right] \cap Q_d',$$
$$V^{j_1,\dots,j_{d-1}}_k:= \left[ \left(\frac {j_1}{2^{k}}, \frac {j_1+1}{2^{k}} \right)\times\dots\times \left(\frac {j_{d-1}}{2^{k}}, \frac {j_{d-1}+1}{2^{k}} \right) \times \left(0,1 \right) \right] \cap Q_d'$$
be the {\it horizontal dyadic slabs and vertical dyadic strips of width $2^{-k}$}, respectively.   Finally, let 
\[
G^d_{k,k'}:= \left\{H^{j_d}_k\cap V^{j_1,\dots,j_{d-1}}_{k'} \,\Big|\, j_d\in \mathbb{Z}\cap[0,2^{k}) \text{ and } j_1,\dots,j_{d-1} \in \mathbb{Z}\cap[0,2^{k'}) \right\}
\]
 be the collection of all  {\it dyadic boxes} of size $2^{-k'}\times\dots\times 2^{-k'}\times 2^{-k}$
  and let $G^d_k:=G^d_{k,k}$  be the collection of all  {\it dyadic cubes} of size $2^{-k}\times\dots\times 2^{-k}$.
\end{definition}

That is, $T_{d,i}$ performs the transformation $T$ in the $x_ix_d$ plane while all other coordinates are preserved. Therefore, any dyadic box 
is being doubled in all $d-1$ horizontal directions $x_1,\dots,x_{d-1}$ and shrunk by a factor of $2^{d-1}$ in the vertical direction $x_d$ by each repeated application of $T_d$ until it becomes a horizontal dyadic slab.  Continued applications of $T_d$ then always multiply the number of horizontal dyadic slabs by $2^{d-1}$ and shrink their widths by a factor of $2^{d-1}$, while these slabs become fairly regularly distributed throughout $Q_d$.
 This immediately gives the following extension of Lemma \ref{ImageOfHorizontal}.

\begin{lemma}\label{L.2.8}
Let  $k,k',l\in\mathbb{N}$, and let $j_d\in \mathbb{Z}\cap[0, 2^{k})$ and $j_1,\dots,j_{d-1}\in \mathbb{Z}\cap[0, 2^{k'})$.
\begin{enumerate}[(i)]
\item
If $i\in \mathbb{Z}\cap[0, 2^{(d-1)l})$, then $T_d^l(H^{j_d}_k)\cap H^i_{(d-1)l}$ is a single horizontal dyadic slab of width $2^{-(k+(d-1)l)}$.  
\item
If $l\le k'$, then $T_d^{l}(H^{j_d}_k\cap V^{j_1,\dots,j_{d-1}}_{k'})$ is a single horizontal dyadic box from $G^d_{k+(d-1)l, k'-l}$.
\end{enumerate}
\end{lemma}

With this lemma in hand, the remaining mixing results in two dimensions and their proofs easily extend to any dimension.

\begin{lemma}\label{L.2.9}
Lemmas \ref{UniversalDiscrete} and \ref{L.2.2} hold for any $d\ge 2$, with $Q_2$ and $T$ replaced by $Q_d$ and $T_d$.  The same is true for Lemma \ref{L.2.5} if we also replace $\gamma<2$ and $2-\gamma$ by $\gamma<d$ and $d-\gamma$ in (i,ii).
\end{lemma}


\section{Construction of the Relevant Velocity Fields}\label{S3}

Let us first consider the no-flow boundary conditions case.  In this case we will construct time-periodic velocity fields realizing the maps $T_d$ from the previous section as their flow maps at the time equal to a single period.  Later we will show how to modify our construction when the boundary conditions are periodic.

Again we will start with the $d=2$ case.  Note that the crucial map $T$ is obtained by a $90^\circ$ rotation of the right and left halves of $Q_2$ in the clockwise and counter-clockwise directions, respectively, followed by a $90^\circ$ counterclockwise rotation of $Q_2$.  We therefore just need to find a divergence free field $u$ that rotates a square by $90^\circ$, and then easily adjust it so it rotates rectangles.  This was done by Yao and the second author in \cite{YaoZla}, but we will redo and slightly alter their construction here as we also want to show that $u(\cdot,t)\in W^{s,p}(Q_2)$ for some $s>1$.  Additionally, since the rotating flow will be time-independent, we will omit $t$ below.

\subsection{Rotating velocity fields on $Q_2$}


For $\alpha\ge 0$, let us consider the stream function $\psi_\alpha:Q_2\rightarrow \mathbb{R}$ given by
\[
\psi_\alpha(x,y) := 2^\alpha \frac{\sin(\pi x)\sin(\pi y)}{(\sin(\pi x)+\sin(\pi y))^{\alpha}}\in \left(0, 1 \right]
\] 
It is easy to show that $\log \psi_\alpha$ is concave (see Lemma \ref{L.3.1}(i) below), so the level sets of $\psi_\alpha$ are curves which foliate $Q_2$. We also define the quantity 
\[
T_{\psi_\alpha}(r):=\int_{\{\psi_\alpha=r\}}\frac{1}{|\nabla\psi_\alpha|}d\sigma,
\] 
which is the time it takes a particle from the level curve $\{\psi_\alpha=r\}$ to traverse this level curve if it moves with  the (divergence-free) velocity $v_\alpha:=\nabla^\perp\psi_\alpha$.  As in \cite{YaoZla}, we will next find another function $\psi^\alpha$ with the same level sets as $\psi_\alpha$ but with the ``period'' $T_{\psi^\alpha}(r)$ independent of $r$ (e.g., equal to 1).  Four-fold rotational symmetry of $\psi_\alpha$ around $(\frac{1}{2},\frac{1}{2})$ will then show that the flow $v^\alpha=\nabla^\perp\psi^\alpha$ rotates $Q_2$ by $90^\circ$ after a quarter of the period $T_{\psi^\alpha}$.

Let us start with some properties of $\psi_\alpha$.

\begin{lemma} \label{L.3.1}
For any $\alpha\in (0,1)$ and $k\in\mathbb{N}\cup\{0\}$ there is $C_{\alpha,k}>0$ such that the following hold.
\begin{enumerate}[(i)]
\item $\log \psi_\alpha$ is concave, so the super-level sets of $\psi_\alpha$ are convex.
\item $|D^k\psi_\alpha(x,y)|\leq C_{\alpha,k} (\sin(\pi x)+\sin(\pi y))^{2-\alpha-k}$ for each $(x,y)\in Q_2$.
\item $T_{\psi_\alpha}\in C^2_{\rm loc}((0,1))$, and for $k=0,1,2$ and each $r\in (0,1)$ we have
\[
\left| T^{(k)}_{\psi_\alpha}(r) \right| 
\leq C_{\alpha,k} \,\, (1-r)^{-k} (1+r^{\frac{\alpha}{2-\alpha}-k}). 
\] 
\end{enumerate}
\end{lemma}

\begin{proof}
(i) is a direct calculation (see the appendix) and   
(ii) is immediate from the definition.

(iii) Since $\psi$ is smooth and $|\nabla\psi|$ is non-zero away from $\partial Q_2\cup\{(\frac 12,\frac 12)\}$, we only need to consider $s$ close to $0$ and $1$.

For all $s$ close to $1$ we have
 \begin{equation} \label{3.1}
\sqrt {1-r} \approx |\{\psi_\alpha=r\}| \approx |\nabla\psi_\alpha(x,y)|\approx \left|(x,y)-\left(\frac{1}{2}, \frac{1}{2} \right) \right| \qquad \text{ for all $(x,y)\in\{\psi_\alpha=r\}$},
 \end{equation}
where $A\approx B$ means that there exists a constant $C$, independent of $(x,y)$ as well as of $s$ near $1$, such that $\frac{1}{C}A\leq B\leq C A$.  Indeed, this follows from $\psi_\alpha$ having a maximum at $(\frac{1}{2},\frac{1}{2})$ with $D^2\psi_\alpha(\frac{1}{2},\frac{1}{2})$ a non-zero multiple of the identity matrix, and from uniform bounds on the third derivatives of $\psi_\alpha$ near $(\frac{1}{2},\frac{1}{2})$. 
It implies, in particular, uniform boundedness of $T_{\psi_\alpha}$ near $1$, which yields the claim for $k=0$ and all $r$ near $1$. 

Let us next study the derivatives of $T_{\psi_\alpha}$ near $1$.  As in \cite{YaoZla}, let us write
\[
T_{\psi_\alpha}(r)=\int_{\{\psi_\alpha=r\}}\frac{1}{|\nabla\psi_\alpha|}d\sigma=\int_{\{\psi_\alpha=r\}}(-n)\cdot \frac{\nabla\psi_\alpha}{|\nabla\psi_\alpha|^2}d\sigma=-\int_{\psi_\alpha>r}\nabla\cdot\left( \frac{\nabla\psi_\alpha}{|\nabla\psi_\alpha|^2}\right) dxdy.
\]
Thus 
\[
T_{\psi_\alpha}'(r)=\int_{\{\psi_\alpha=r\}}\frac{1}{|\nabla\psi_\alpha|} \nabla\cdot\left(\frac{\nabla\psi_\alpha}{|\nabla\psi_\alpha|^2}\right)d\sigma.
\]
From this, \eqref{3.1}, and (ii) for $k=2$ we now obtain
for all $r$ near $1$,
\[
 \left|T'_{\psi_\alpha}(r) \right| 
 \leq C_{\alpha}(1-r),
\] 
with some constant $C_{\alpha}$ (depending only on $\alpha$).  Similarly we obtain
\[
T_{\psi_\alpha}''(r)=\int_{\{\psi_\alpha=r\}}\frac{1}{|\nabla\psi_\alpha|}\nabla\cdot\left(\frac{\nabla\psi_\alpha}{|\nabla\psi_\alpha|^2} \nabla\cdot\left(\frac{\nabla\psi_\alpha}{|\nabla\psi_\alpha|^2}\right) \right)d\sigma,
\] 
and then  for all $r$ near $1$,
\[ 
\left| T''_{\psi_\alpha}(r)\right| 
\leq C_{\alpha}' (1-r)^2,
\]
with some constant $C_{\alpha}'$.  Hence the claim for $k=1,2$ and all $r$ near $1$ also follows.

Let us now consider $r$ near 0.  A simple computation gives the lower bound 
\[
|\nabla\psi_\alpha(x,y)|\geq c_\alpha (\sin(\pi x)+\sin(\pi y))^{1-\alpha}
\] 
for all such $r$ and all $(x,y)$ with $\psi_\alpha(x,y)=r$. In particular, 
\[
T_{\psi_\alpha}(r)\leq \int_{\{\psi_\alpha=r\}}\frac{d\sigma}{c_\alpha(\sin(\pi x)+\sin(\pi y))^{1-\alpha}}.
\]
Since $\psi_\alpha$ and $\sin(\pi x)+\sin(\pi y)$ have all the symmetries of $Q_2$, it suffices to consider 8 times the last integral restricted to 
the part of the curve $\{\psi_\alpha=r\}$ between the lines $y=x$ and $x=\frac{1}{2}$ (where $\sin(\pi x)\ge\sin(\pi y)$).   For $r$ near 0 and $x$ near 0 such that  $\psi_\alpha(x,x)=r$ we have $\sin(\pi x)^{2-\alpha}= r$, so that $x\ge \frac 1{10}r^{\frac{1}{2-\alpha}}$.  It follows that for some $C_\alpha$ and all $r$ near 0 we have
\[
T_{\psi_\alpha}(r)\leq C_\alpha \int_{0.1r^{\frac{1}{2-\alpha}}}^{\frac{1}{2}} \frac{ dx}{x^{1-\alpha}}\leq \frac{C_\alpha}\alpha,
\]
which yields the claim for $k=0$ and all $r$ near 0.
 
We now proceed as in the case $r$ near 1. From
\[
T_{\psi_\alpha}'(r)=\int_{\{\psi_\alpha=s\}}\frac{1}{|\nabla\psi_\alpha|} \nabla\cdot\Big(\frac{\nabla\psi_\alpha}{|\nabla\psi_\alpha|^2}\Big) d\sigma
\]
and (ii) we obtain for $r$ near 0 (with some constants $C_\alpha, C_\alpha', C_\alpha''$),
\[
|T_{\psi_\alpha}'(r)|\leq \int_{\{\psi_\alpha=r\}}C_{\alpha}\frac{(\sin(\pi x)+\sin(\pi y))^{-\alpha}}{|\nabla\psi_\alpha|^3} d\sigma \leq C_\alpha' \int_{0.1 r^{\frac{1}{2-\alpha}}}^{\frac{1}{2}}\frac{1}{(\sin(\pi x))^{3-2\alpha}} d\sigma \le C_{\alpha}'' r^{\frac{\alpha}{2-\alpha}-1}.
\]
Similarly,
from the formula for $T_{\psi_\alpha}''$, lower bound on $|\nabla\psi_\alpha|$, and (ii) we obtain for $r$ near 0,
\[
|T_{\psi_\alpha}''(r)|\leq C_{\alpha} \int_{0.1 r^{\frac{1}{2-\alpha}}}^{\frac{1}{2}} \frac{dx}{(\sin(\pi x))^{5-3\alpha}}\le C_{\alpha}' r^{\frac{\alpha}{2-\alpha}-2}.
\] 
Hence the claim for $k=1,2$ and all $r$ near $0$ also follows, and we are done.
We note that one can continue to higher derivatives and, in particular, obtain that $T_{\psi_\alpha}\in C^\infty_{\rm loc}((0,1))$.
\end{proof}

We now define the function $\psi^\alpha$ with the same level sets as $\psi_\alpha$ but with $T_{\psi^\alpha}(r)$ independent of $r$.  Because we want to address the question of its (fractional) regularity, let us first define the fractional Sobolev spaces.

\begin{definition}
For $\gamma\in(0,1)$ and for $f\in C^1(Q_2)$, let 
\[
\Lambda^\gamma(f)(x,y)=\int_{Q_2}\frac{f(x,y)-f(x',y')}{|(x,y)-(x',y')|^{2+\gamma}}dx'dy'.
\] 
 For $p\ge 1$, define the $W^{\gamma,p}$ norm of $f$ by $\|f\|_{W^{\gamma,p}(Q_2)}:=\|f\|_{L^p(Q_2)} + \|\Lambda^\gamma(f)\|_{L^p(Q_2)}.$ 
The space $W^{\gamma,p}(Q_2)$ is the completion of $C^1(Q_2)$ with respect to this norm. We also let $W^{0,p}(Q_2):=L^p(Q_2)$ and for $\gamma\ge 1$ we say that $f\in W^{\gamma,p}(Q_2)$ if and only if $f, Df,\dots,D^{\lfloor\gamma\rfloor}f\in W^{\gamma-\lfloor\gamma\rfloor,p}(Q_2).$ 
\end{definition}

\begin{lemma} \lb{L.3.2}
For any $\alpha\in (0,1)$ there exists $\psi^\alpha\in W^{1,\infty}(Q_2)\cap W^{2,\infty}_{\rm loc}(Q_2) $ with the following properties.
\begin{enumerate}[(i)]
\item $\psi^\alpha$ has the same level sets as $\psi_\alpha$, with $\psi^\alpha>0$ on $Q_2$ and $\psi^\alpha=0$ on $\partial Q_2$.
\item $T_{\psi^\alpha}(r)=1$ for all $r\in (0,\|\psi^\alpha\|_{L^\infty}).$
\item For any $\gamma\in[0,1]$ we have $\psi^\alpha\in W^{2+\gamma,p}(Q_2)$ whenever $1\le p<(\max\{\frac{\alpha}{2},\frac{2-2\alpha}{2-\alpha}\}+\gamma)^{-1}$.
\end{enumerate} 
\end{lemma}

{\it Remark.}
Optimal regularity in (iii) is thus obtained when $\frac{\alpha}{2}=\frac{2-2\alpha}{2-\alpha}$, that is, by taking $\alpha_*:=3-\sqrt{5}.$  This yields $\psi^{\alpha_*}\in {W^{2+\gamma,p}(Q_2)}$ for all $\gamma\in[0,\frac{\sqrt 5 -1}2)$ and $p\in[1,\frac{2}{3-\sqrt{5}+2\gamma}).$ 

\begin{proof}
Let
\beq \lb{3.1a}
\psi^\alpha(x,y):=\int_0^{\psi_\alpha(x,y)}T_{\psi_\alpha}(r)dr,
\eeq
from which (i) follows immediately.
Since $\nabla\psi^\alpha(x,y)= T_{\psi_\alpha}(\psi_\alpha(x,y)) \nabla\psi_\alpha(x,y)$, Lemma \ref{L.3.1} yields $\psi^\alpha\in W^{1,\infty}$ and we also obtain (ii).
It remains to show (iii).
 
We see that \[|D^2\psi^\alpha|\leq |D^2\psi_\alpha| T_{\psi_\alpha}(\psi_\alpha)+|\nabla\psi_\alpha|^2|T'_{\psi_\alpha}(\psi_\alpha)|,\]
and \[|D^3\psi^\alpha|\leq |D^3\psi_\alpha|T_{\psi_\alpha}(\psi_\alpha)+3|D^2\psi_\alpha||\nabla\psi_\alpha||T'_{\psi_\alpha}(\psi_\alpha)|+|\nabla\psi_\alpha|^3 |T''_{\psi_\alpha}(\psi_\alpha)|.\] Lemma \ref{L.3.1} and \eqref{3.1} show that away from $\partial Q_2$, the only unbounded terms may be the second and third term in the estimate for $D^3\psi^\alpha$, and they are both bounded above by ${C_\alpha}|(x,y)-(\frac 12,\frac 12)|^{-1}$ there.
 Thus $D^2\psi^\alpha$ is bounded away from $\partial Q_2$ (proving $\psi^\alpha\in W^{2,\infty}_{\rm loc}$), while $D^3\psi^\alpha$ is in $L^p$ for any $p<2$ there.
Sobolev embedding now shows that $D^2\psi^\alpha\in W^{\gamma,p}$ away from $\partial Q_2$ whenever $p<\frac 2\gamma$.

It therefore suffices to consider the neighborhood of $\partial Q_2$.
From Lemma \ref{L.3.1} we see that there
\[
|D^2\psi^\alpha(x,y)|\leq C_\alpha\Big(|\sin(\pi x)+\sin(\pi y)|^{-\alpha}+|\psi_\alpha|^{\frac{2\alpha-2}{2-\alpha}}\Big) \leq 2C_\alpha |\sin(\pi x)\sin(\pi y)|^{-\max\{\frac{\alpha}{2},\frac{2-2\alpha}{2-\alpha}\}}
\] 
for some $C_\alpha$.
Similarly, near $\partial Q_2$ we obtain 
\[
|D^3\psi^\alpha(x,y)|\leq C_\alpha|\sin(\pi x)\sin(\pi y)|^{-\max\{\frac{\alpha+1}{2}, \frac \alpha 2 + \frac{2-2\alpha}{2-\alpha}, \frac{4-3\alpha}{2-\alpha}\}}=C_\alpha|\sin(\pi x)\sin(\pi y)|^{-\frac{4-3\alpha}{2-\alpha}}
\] 
because $\alpha\in[0,1]$.  
Corollary \ref{fractionalsines} applied to $D^2\psi^\alpha$  now yields
\[
|\Lambda^\gamma D^2\psi^\alpha(x,y)|\leq C_{\alpha,\gamma,\epsilon}|\sin(\pi x)\sin(\pi y)|^{-\max\{\frac{\alpha}{2},\frac{2-2\alpha}{2-\alpha}\}-\gamma-\epsilon}
\] 
for all $\epsilon>0$ and $(x,y)$ near $\partial Q_2$,
with some $C_{\alpha,\gamma,\epsilon}$.
Direct integration and the estimate away from $\partial Q_2$ (with the range from (iii) included in $p<\frac 2\gamma$) finish the proof of (iii). 
\end{proof}

\subsection{Velocity fields realizing $T_d$ as their flow map (no-flow case)} \lb{S3.2}

According to the remark after Lemma \ref{L.3.2}, let us take $\alpha_*:=3-\sqrt{5}$ and let $\psi:=\frac12 \psi^{\alpha_*}$.  Then Lemma \ref{L.3.2} shows that the velocity field $v:=\nabla^\perp \psi=(-\psi_y,\psi_x)$ belongs to $W^{s,p}(Q_2)$ for any $s<\frac{1+\sqrt 5}2$ and $p\in[1,\frac{2}{2s+1-\sqrt{5}})$ and it rotates $Q_2$ clockwise by $90^\circ$ in time $\frac 12$.  Similarly, if $\phi(x,y):= \frac 12 {\rm sgn}(2x-1)\psi(\{2x\},y)$ (with $\{x\}$ the fractional part of $x$) and $w:=\nabla^\perp \phi$, then $w$  ``rotates'' the left and right halves of $Q_2$ counterclockwise and clockwise by $90^\circ$ in time $\frac 12$, respectively (by ``rotation'' we mean the affine map that is a bijection on the rectangle and permutes its corners in the indicated direction).  It therefore follows that
\beq\lb{3.2}
u(x,y,t):=
\begin{cases}
w(x,y) & \{t\}\in[0,\frac 12)
\\ -v(x,y) & \{t\}\in[\frac 12,1)
\end{cases}
\eeq
is time 1-periodic, satisfies the no-flow boundary condition, and its flow map at time 1 is $T$ from \eqref{2.0}.  Moreover, $w$ also belongs to  $W^{s,p}(Q_2)$ for any $s<\frac{1+\sqrt 5}2$ and $p\in[1,\frac{2}{2s+1-\sqrt{5}})$.  Indeed,  $\phi$ is Lipschitz continuous across $\{x=\frac 12\}$ since $\psi$ vanishes at $\{x=1\}$.  Also, $\phi_x$ is clearly continuous across $\{x=\frac 12\}$ while the same is true for $\phi_y$ because it vanishes there.   This and $\psi\in W^{2,p}(Q_2)$ now yield $\phi\in W^{2,p}(Q_2)$, and Corollary \ref{fractionalsines} applied to $D^2 \phi$ proves the claim.

Of course, the application of $v$ does nothing for mixing and one could instead replace it by $R^{-1}w(R(x,y))$ for times $t$ with $\{t\}\in[\frac 12,1)$, where $R$ is the counterclockwise rotation of $Q_2$ by $90^\circ$.  Then the time-$\frac 12$ and time-1 flow maps of the new 1-periodic flow will be $R^{-1}T$ and $R^{-2}T^2$, respectively, but an analog of Lemma \ref{ImageOfHorizontal} holds in this case and so do other results in Section 2.

It is now also clear how to construct the relevant velocity fields in higher dimensions.  They will have time periods $d-1$ and have the above $u$ acting in only 2 variables (specifically, 1 and $d$, 2 and $d$,..., $d-1$ and $d$) on each time interval with integer endpoints.  These time-periodic and time-piecewise-constant fields obviously again belong to $W^{s,p}(Q_d)$  for all $(s,p)$ from Theorem~\ref{T.1.1}(i).  Of course, we can then scale them in time and value by $d-1$ to obtain time period 1.

\subsection{Velocity fields on $\bbT^d$ (periodic case)}

In the case of periodic boundary conditions we cannot use the flow $u$ from the previous subsection because it is not $W^{s,p}$ across $\partial(0,1)^2$.  A way to fix this is to extend the stream functions $\psi,\phi$ oddly in both $x$ and $y$ onto $(-1,1)^2$ and then map them back onto $Q_2$ (or rather $\bbT^2$).  So we let
\[
\psi'(x,y):= \frac 12 \psi (2x-1,2y-1)
\]
(and similarly for $\phi'$) where $\psi$ is the doubly-odd extension.  If we now let $v':=\nabla^\perp \psi'$ and $w':=\nabla^\perp \phi'$, then these fields again belong to $W^{s,p}(\bbT^2)$ for any $s<\frac{1+\sqrt 5}2$ and $p\in[1,\frac{2}{2s+1-\sqrt{5}})$ by the argument from the previous subsection.  Now each of the four square cells from $G_1$ (of side length $\frac 12$) is left invariant by $v'$ and $w'$.  Therefore, a flow like \eqref{3.2} can mix exponentially quickly within each of these four cells, but there is no ``exchange'' between the cells.  

We remedy this problem by instead letting
\beq\lb{3.3}
u'(x,y,t):=
\begin{cases}
(\frac 12,\frac 12) & \{t\}\in[0,\frac 12)
\\ w'(x,y) & \{t\}\in[\frac 12,\frac 34)
\\ -v'(x,y) & \{t\}\in[\frac 34,1),
\end{cases}
\eeq
which satisfies periodic boundary conditions and at each time belongs to all the $W^{s,p}(\bbT^2)$ spaces above.
We denote by $T'$ the time-1 flow map of $u'$.  Let us also define
\[
\tilde T(x,y)=\begin{cases} 
     - (2x, \frac{y}{2}) + (1,1) & x\in(0,\frac{1}{2}), \\
     (2x, \frac{y}{2}) + (-1,0) & x\in (\frac{1}{2},1),   \end{cases}
\]
a mapping closely related to $T$ from \eqref{2.0}, and recall that $R^2$ is the $180^\circ$ degree  rotation of $Q_2$.

\begin{figure}[htbp]
\centerline{\includegraphics[scale=0.6]{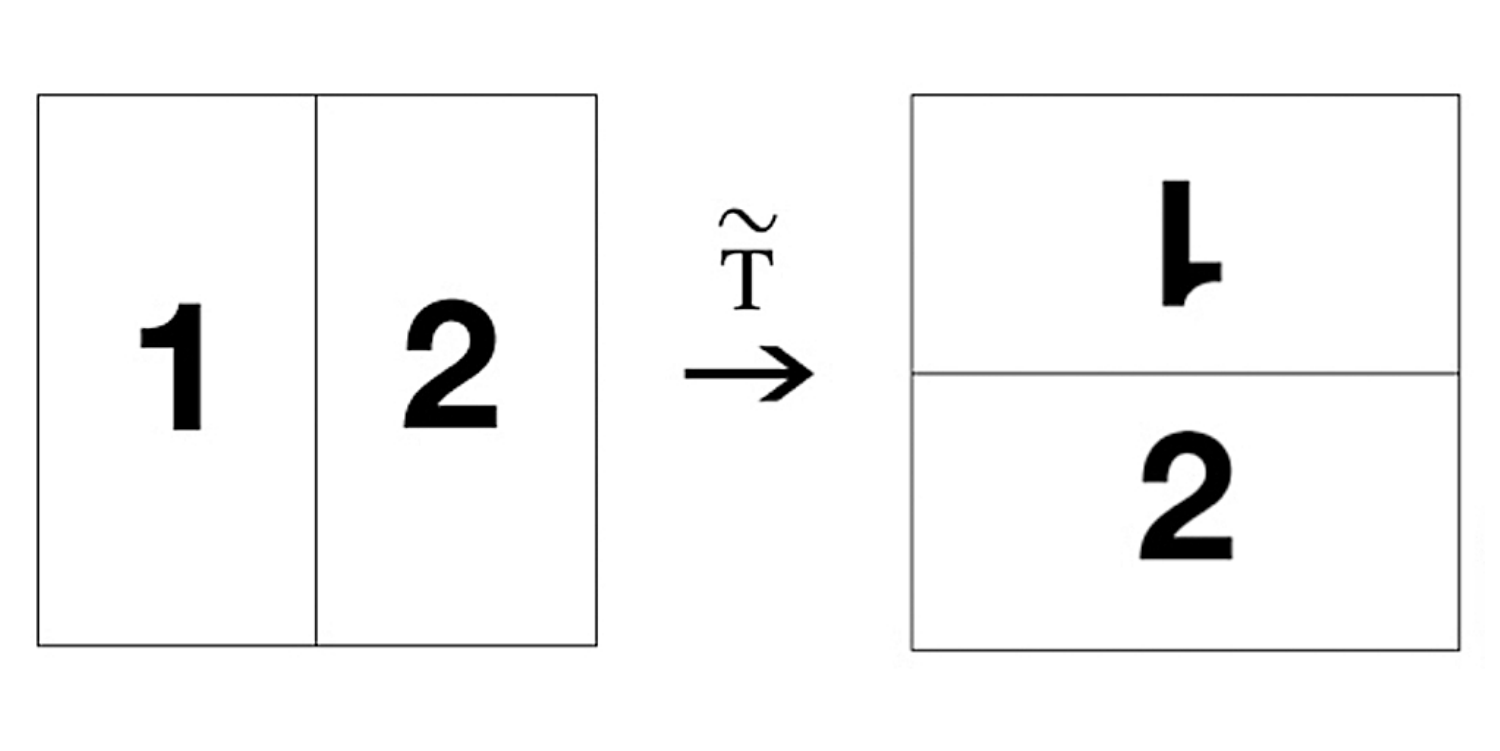}}
\caption{The mapping $\tilde T$.}\label{fig2}
\end{figure}

Consider now any dyadic rectangle $H^j_k\cap V^{j'}_{k'}\in G_{k,k'}$ with $k\ge 2$.  
If $k'\ge 2$, then its translation by $(\frac 14,\frac 14)$ that occurs under the action of $u'$ over time interval $[0,\frac 12)$ transforms $H^j_k\cap V^{j'}_{k'}$ into another element of $G_{k,k'}$, which is now fully contained inside one of the four cells from $G_1$.
If this cell is $(0,\frac 12)\times(0,\frac 12)$, $(\frac 12,1)\times(0,\frac 12)$, $(0,\frac 12)\times(\frac 12,1)$, or  $(\frac 12,1)\times(0,\frac 12)$, then the above action of $w'$ and $v'$ over time interval $[\frac 12,1)$ on this cell  is the same as the action of $T$, $R^2T$, $\tilde T$, or $R^2\tilde T$ on $Q_2$, respectively (due to the factor $\frac 12$ in the definition of $\psi'$ and $\phi'$).  Therefore, $T'(H^j_k\cap V^{j'}_{k'})\in G_{k+1,k'-1}$.

If instead $k'=1$, then the translation by $(\frac 14,\frac 14)$ splits the dyadic rectangle between two horizontally adjacent cells, but a direct computation (or the geometric picture described above) shows that again $T'(H^j_k\cap V^{j'}_{k'})\in G_{k+1,0}$.  It follows that Lemma \ref{ImageOfHorizontal}(ii) continues to hold for $T'$ in place of $T$ as long as $k\ge 2$.  

If now $H^j_k\in G_{k,0}$ with $k\ge 2$, then $T'(H^j_k)$ consists of two horizontal dyadic strips of width $2^{-(k+1)}$ located either in $H^0_2$ and $H^1_2$ or in $H^2_2$ and $H^3_2$.  In either case, $(T')^2(H^j_k)$ will then consists of four horizontal dyadic strips of width $2^{-(k+2)}$, one contained in each of $H^0_2$, $H^1_2$, $H^2_2$, and $H^3_2$ (a setup preserved by a translation by $(\frac 14,\frac 14)$).  A repeated application of $T'$ then shows that Lemma \ref{ImageOfHorizontal}(i) continues to hold for $T'$ in place of $T$ as long as $k,l\ge 2$.

This modified Lemma \ref{ImageOfHorizontal} is of course sufficient for the rest of the analysis in Section 2.  This includes the case $d\ge 3$, where we construct the relevant mapping $T_d'$ using $T'$ in the same way we constructed $T_d$ using $T$ (and then even both parts of Lemma \ref{L.2.8} will hold for $T_d'$ in place of $T_d$ as long as $k\ge 2$).
Thus there is again  a time 1-periodic and time-piecewise-constant vector field that belongs to $W^{s,p}(\bbT^d)$ for all $(s,p)$ from  Theorem \ref{T.1.1}(i) and whose time-1 flow map is $T_d'$.

\subsection{Non-integer times and the proof of Theorem \ref{T.1.1}} 

Theorem \ref{T.1.1}(ii) immediately follows from Lemma \ref{L.2.4}, with $T_n$ the flow map of the vector field in question at time $n$.

Theorem \ref{T.1.1}(i) for only integer times follows from the above constructions and from Lemmas~\ref{L.2.2} and \ref{L.2.5}(iii) (as well as their multidimensional analogs in Lemma \ref{L.2.9} and their periodic boundary conditions analogs from the previous subsection).  

Let us now extend the first claim in Theorem \ref{T.1.1}(i) to all times (we only do it on $Q_2$, the other cases are analogous).  This uses the fact that the time-independent flows $w$ and $-v$ from \eqref{3.2} are bounded, due to Lemma \ref{L.3.1}.
This implies that their actions are continuous in time on $L^2$, that is, if $\{\Phi_{t;z}\}_{t\ge 0}$ are their flow maps (with $z\in\{v,w\}$) and $g\in L^2(Q_2)$, then both maps $t\mapsto g\circ\Phi_{t;z}^{-1}$ belong to $C([0,\infty),L^2(Q_2))$.  Let also $\{\Phi_t\}_{t\in[0,1]}$ be the flow maps for the (time-dependent) flow $u$ from \eqref{3.2}.  Note that $\Phi_{-t}\neq\Phi_t^{-1}$ due to time dependence, and in fact we have
\[
\Phi_t=
\begin{cases}
\Phi_{t;w} & t\in[0,\frac 12],
\\ \Phi_{t-1/2;v}\circ \Phi_{1/2;w} & t\in(\frac 12,1],
\\   \Phi_{t;v}   & t\in[-\frac 12,0),
\\  \Phi_{t+\frac 12;w}  \circ \Phi_{-1/2;v}  & t\in[-1,-\frac 12).
\end{cases}
\]

Let a bounded mean-zero $\rho$ solve \eqref{1.1} with the flow $u$ on $Q_2$ from \eqref{3.2} and let $f:=\rho(\cdot,0)$.  Then $\rho(\cdot,n+s)=f\circ T^{-n}\circ\Phi_s^{-1}$ for each $n\in\bbN$ and $s\in[0, 1]$.
Given any $g\in L^2(Q_2)$ and $\eps>0$, let $N\in\bbN$ be such that for $z\in\{v,w\}$ we have
\[
\sup_{|s|\le 1/N}\|g\circ\Phi_{s;z}^{-1} - g\|_{L^2}\le \eps.
\]
Let $g_j:=g\circ\Phi_{-j/N}^{-1}\in L^2(Q_2)$ for $j=0,\dots,N-1$.  The weak-$L^2$ convergence in the proof of Lemma \ref{L.2.2} shows that $\{\rho(\cdot, \frac mN)\}_{n\ge 0}$ also converges weakly to 0 in $L^2(Q_2)$, because if $m=nN+j$ with $j\in\{0,\dots,N-1\}$, then incompressibility of $u$ shows that
\[
\int_{Q_2} \rho\left(\cdot, \frac mN \right) g dx = \int_{Q_2} \left(f\circ T^{-n}\circ \Phi_{j/N}^{-1} \right) g dx = \int_{Q_2} \left( f\circ T^{-n} \right) g_j dx \qquad \text{($\to 0$ as $m\to\infty$).}
\]
Now if $t=\frac mN+s$ with $s\in[0,\frac 1N)$, then
\[
\int_{Q_2} \rho\left(\cdot, t \right) g dx =  \int_{Q_2}  \rho\left(\cdot, \frac mN \right) \left( g\circ\Phi_{-s;z}^{-1}\right) dx,
\]
where $z=w$ if $\lfloor \frac mN \rfloor<\frac N2$ and $z=v$ otherwise (recall that $N$ is even).  But then 
\[
\left| \int_{Q_2} \rho\left(\cdot, t \right) g dx \right| \le \left| \int_{Q_2}  \rho\left(\cdot, \frac mN \right)gdx\right| + \left\| \rho \left(\cdot,\frac mN \right) \right\|_{L^2}^{1/2} \| g\circ\Phi_{-s;z}^{-1} -g \|_{L^2}^{1/2}.
\]
The first term on the right-hand side converges to 0 as $t\to \infty$ (because then $m\to\infty$), while the second is at most $ \|\rho(\cdot,0)\|_{L^2}^{1/2}\eps^{1/2}$.  Since  $g\in L^2(Q_2)$ and $\eps>0$ were arbitrary, the claim of asymptotic mixing in the functional sense is proved.  The same claim in the geometric sense again follows from this and Lemma A.1 in \cite{YaoZla}.

The next lemma, which  is of independent interest, shows that the flow maps of our square-rotating flow for $t\in[0,1]$ (so until the square is rotated by $90^\circ$) are uniformly H\" older continuous.  (Note that since the flow is not Lipschitz, H\" older continuity of the flow maps is not obvious.)  This will extend the exponential mixing claim  in Theorem \ref{T.1.1}(i) from integer times to all times, as we show next. 
For the sake of convenience, we switch to the notation $x=(x_1,x_2)$ on $Q_2$ in the rest of this subsection.

\begin{lemma} \lb{L.3.4}
Let $\alpha\in(0,1)$ 
and consider the velocity field $u':=\nabla^\perp \psi^{\alpha}$ on $Q_2$, with $\psi^\alpha$ from \eqref{3.1a}.  Then the flow maps $\Phi'(\cdot,t)$ for $u$, given by
\[
\frac d{dt} \Phi'(x,t) = u(\Phi'(x,t)) \qquad\text{and}\qquad \Phi'(x,0)=x,
\]
are H\" older continuous uniformly in $t\in[0,1]$ (with some exponent $\beta_\alpha>0$ and constant $C_\alpha$).
\end{lemma}

Let us now prove the second claim in Theorem \ref{T.1.1}(i).
Consider any $\sigma>0$ and $\rho(\cdot,0)\in H^\sigma(Q_2)$.  Let a bounded mean-zero $\rho$ solve \eqref{1.1} with the flow $u$ on $Q_2$ from \eqref{3.2} and let $f:=\rho(\cdot,0)$.  Then, $\rho(\cdot,n+t)=f\circ T^{-n}\circ\Phi_t^{-1}$ for $n\in\bbN$ and  $t\in[0,\frac 12]$, where now $\Phi_t$ is the flow map at time $t$ for the time-independent flow $w$ from \eqref{3.2} (denoted $\Phi_{t;w}$ above).
 By Lemma \ref{L.3.4}, these flow maps are uniformly H\" older continuous (with some exponent $\beta>0$ and constant $C$) separately on the left and right halves of $Q_2$.  Assume $\beta\le \frac 23$ wihtout loss.

 Let $f_k$ be from the proof of Lemma \ref{L.2.5}(iii) and let $Q'\in G_{\lfloor \delta k\rfloor}$ be arbitrary, with some $\delta>0$.   
 We will show that estimate \eqref{2.3} for squares $Q\in G_k$ and the H\" older bound on $\Phi_t$ together yield an analogous estimate for $f_k\circ T^{-2k}\circ\Phi_t^{-1}$ on $Q'$ with any $t\in[0,\frac 12]$.  Indeed, assume without loss that $Q'$ belongs to the left half $L$ of $Q_2$ (on which the flow maps are H\" older continuous) and let $Q_t:=\Phi_t^{-1}(Q')\subseteq L$.   Let $P_t\subseteq Q_t$ be the union of all the squares from $G_k$ that are fully contained in $Q_t$.  The H\" older bound shows that if $x\in Q_t\setminus P_t$, then ${\rm dist}(\Phi_t(x),\partial Q')\le C2^{\beta(1-k)}$.  That is, $|Q_t\setminus P_t|\le 4C2^{\beta - \lfloor \delta k\rfloor-\beta k}$.  Picking $\delta\le\frac \beta 2$ shows that $|Q_t\setminus P_t|\le C_1 2^{-\beta k/2} |Q_t|$ with some constant $C_1$.  This, \eqref{2.3}, and $\Phi_t$ preserving measure now show (with some constants $C_2,C_3$)
 \[
 \left| \fint_{Q'} f_k\circ T^{-2k}\circ\Phi_t^{-1} dx \right| =  \left| \fint_{Q_t} f_k\circ T^{-2k} dx \right| \le C_22^{-k/3}\|f\|_{H^\sigma} + C_1 2^{-\beta k/2} \|f\|_{L^\infty} \le C_3 2^{-\beta k/2} \|f\|_{H^\sigma}
 \]
 (and the same bound again holds with $T^{-(2k+1)}$ in place of $T^{-2k}$).
 The claim in \eqref{2.4} continues to hold if we replace $T^{-2k}$ by $T^{-2k}\circ\Phi_t^{-1}$, which then  yields exponential decay of $\|\rho(\cdot,t)\|_{\dot H^{-1/2}}$ when restricted to $t\ge 0$ with $\{t\}\in[0,\frac 12]$, with a $\sigma$-dependent rate (this again uses equivalence of the $\dot H^{-1/2}$-norm and the mix norm).  To handle the times with $\{t\}\in(\frac 12,1)$, it suffices to notice that $\Phi_{1/2}$ simply rotates the two halves of $Q_2$ by $90^\circ$ (so it is Lipschitz with constant 2), and the flow $-v$ from \eqref{3.2} satisfies the same H\" older estimate as $w$.  Hence the above argument applies again and almost universal exponential mixing in the functional sense on $Q_2$ follows.  The same claim in the geometric sense again follows from this and Lemma A.1 in \cite{YaoZla}.
The other cases of $Q_d$ and $\bbT^d$ are again analogous because the flows involved are essentially  the same as on $Q_2$.

To finish the proof of Theorem \ref{T.1.1}(i),
it therefore remains to prove Lemma \ref{L.3.4}.

\begin{proof}[Proof of Lemma \ref{L.3.4}]
Let us first  consider $u:=\nabla^\perp \psi_{\alpha}$ instead, and prove the claim about the corresponding flow map $\Phi$ and with $t$ in any interval $[0,\tau]$.  In fact, for the sake of convenience, let us instead consider the equivalent case of 
\[
u(x,y) := \nabla^\perp \tilde \psi_\alpha(x):=\nabla^\perp \frac{\sin x_1 \sin x_2}{(\sin x_1+\sin x_2)^{\alpha}}
\] 
in the square $\pi Q_2$.

Since $u$ is Lipschitz away from the corners, it suffices to consider the case of points $x,y\in \pi Q_2$ such that both trajectories $\{\Phi(x,s)\}_{s\in[0,t]}$ and $\{\Phi(x,s)\}_{s\in[0,t]}$ remain within some distance $0<\delta \ll 1$ from a corner of $\pi Q_2$.  By symmetry, this corner can be assumed to be the origin, so in particular $x,y\in A_\del:= B_\delta(0)\cap \pi Q_2$.   We note that $\delta$ will depend on $\alpha$ (then any constants depending on $\delta$ will really only depend on $\alpha$) and will be such that $\inf_{z\in(0,\delta)} (\frac {\sin z}{z}+\cos z)$ is sufficiently close to 2.  We will split the argument into several cases.

{\bf Case 1.} We will first consider points $x,y\in A_\del$ near the $x_2$ axis, specifically, in the region $S_{\del}:=A_\del\cap\{x_1\le \del x_2\}$.  Their trajectories will be therefore moving primarily upwards and slightly to the left while they remain in $A_\del$ (and therefore they will also remain in $S_\del$ during this time), with the vertical component of $u(x)$ being
\[
u_2(x)= \cos x_1 \sin x_2 \frac{ (1-\alpha)\sin x_1 + \sin x_2 }{(\sin x_1 + \sin x_2)^{1+\alpha}}.
\]
We therefore have
\beq \lb{3.5}
u_2(x) \ge (1-4\del)x_2^{1-\alpha} \qquad \text{for all $x\in S_\del$}
\eeq
 if $\del$ is small enough.  

Similarly, estimating $\nabla u(x)$ yields
\beq \lb{3.5a}
\sup_{x\in S_\del} \max \{|\nabla u_1(x) x_2^{\alpha} - (\alpha-1, 0)| , |\nabla u_2(x) x_2^{\alpha} - (-2\alpha,1-\alpha)| \} \le C\delta
\eeq
if $\del$ is small enough (with a universal constant  $C\ge 1$), and therefore
\[
\sup_{x\in S_\del} \sup_{v\in\bbR^2} |v|^{-2} \left| v\cdot \nabla u(x) v  x_2^{\alpha} - [ (1-\alpha)(v_2^2-v_1^2)-2\alpha v_1v_2] \right| \le C\delta
\]
 (with a new universal $C\ge 1$).  Since the function $(1-\alpha)(v_2^2-v_1^2)-2\alpha v_1v_2$ is 0-homogeneous and strictly below 1, it follows that
\beq \lb{3.6}
|v\cdot \nabla u(x) v| \le (1-5\del) x_2^{-\alpha} |v|^2  \qquad \text{for all $v\in\bbR^2$ and $x\in S_\del$}
\eeq
 if $\del$ is small enough.

Let us now assume that $x,y\in S_\delta$ and $x_2\le y_2\le 2x_2$, and consider any time $t$ such that $\{\Phi(x,s)\}_{s\in[0,t]},\{\Phi(x,s)\}_{s\in[0,t]}\subseteq A_\del$ (and therefore also in $S_\del$). 
It follows from \eqref{3.5} that $\Phi_2(x,s)\ge (\alpha(1-4\del) s + x_2^\alpha)^{1/\alpha}$ for $s\in[0,t]$, and similarly for $y$ in place of $x$, so $x_2\le y_2$ shows that
\[
\min\{\Phi_2(x,s),\Phi_2(y,s)\}\ge (\alpha(1-4\del) s + x_2^\alpha)^{1/\alpha} \qquad\text{for all $s\in[0,t]$.}
\]
This and \eqref{3.6} show that
\beq \lb{3.6a}
\left| \frac d{ds} |\Phi(x,s)-\Phi(y,s)|  \right| \le (1-5\del) (\alpha(1-4\del) s + x_2^\alpha)^{-1}  |\Phi(x,s)-\Phi(y,s)|,
\eeq
and integration then yields
\beq \lb{3.7}
{|\Phi(x,t)-\Phi(y,t)|} \le \left( \frac{\alpha(1-4\del) t + x_2^\alpha}  {x_2^\alpha} \right)^{(1-5\del)/\alpha(1-4\del)}  {|x-y|}.
\eeq
Since $x_2\ge \frac 12|x-y|$ due to $x_2\le y_2\le 2x_2$, and $x_2^\alpha+t\le 1$ if $\del$ is small, it follows that 
\beq \lb{3.8}
|\Phi(x,t)-\Phi(y,t)|   \le 2|x-y|^{1-\frac{1-5\del}{1-4\del}} \le 2|x-y|^\del.
\eeq
This is the desired uniform H\" older bound for the $x,y$ we considered here.

If now $x,y\in S_\delta$ are arbitrary (without loss we can assume $x_2\le y_2$), let $z^0=x, z_1, \dots,z_n=y$ be points on the segment $xy$ such that $x_2\le z^1_2\le 2x_2$ and $z^j_2=2z^{j-1}_2$ for all $j>1$.  Then telescoping \eqref{3.8} applied to couples of points $\{(z^{j-1},z^j)\}_{j=1}^n$ in place of $(x,y)$ yields
\beq \lb{3.9}
|\Phi(x,t)-\Phi(y,t)|   \le  \sum_{j=1}^\infty 2(2^{1-j}|x-y|)^\del  \le 4\del^{-1}|x-y|^\del
\eeq
as long as the trajectories remain in $A_\del$.  This finishes Case 1.

For later reference we  note that $|\Phi(x,\cdot)-\Phi(y,\cdot)|$ also cannot decrease too quickly in $S_\del$ if $\del$ is small enough.  
Specifically, let $\delta\le \frac{(1-\alpha)^2}{4C}$, with $C\ge 1$ from \eqref{3.5a}.  This and  \eqref{3.5a} applied to all points on the segment $\Phi(x,s)\Phi(x,s)$ show that if $v(s):=\Phi(y,s)-\Phi(x,s)$ satisfies $|v_2(t_0)|= \frac {4\alpha}{1-\alpha} |v_1(t_0)|$ at some $t_0\in[0,t]$, then $\frac d{ds}|v_1(t_0)|< 0$ and $\frac d{ds}|v_2(t_0)|> 0$.  That is, we must have $|v_2(s)|\ge \frac {4\alpha}{1-\alpha} |v_1(s)|$ for all $s\in[t_0,t]$, and \eqref{3.5a} now shows that $\frac d{ds}|v_2|>0$ on $[t_0,t]$.  This means that we only need to track $|v(\cdot)|$ only on the longest interval $(0,t_0]\subseteq[0,t]$ on which $|v_2(\cdot)|\le  \frac {4\alpha}{1-\alpha} |v_1(\cdot)|$.  
But then similarly to \eqref{3.6a} we obtain
\[
\left| \frac d{ds} |v_1(s)|  \right| \le  \left(1-\alpha + C\delta+\frac {4\alpha}{1-\alpha}C\delta \right) (\alpha(1-4\del) s + x_2^\alpha)^{-1}  |v_1(s)|=:C'' (\alpha(1-4\del) s + x_2^\alpha)^{-1}  |v_1(s)|
\]
on  $(0,t_0]$.  Integrating this  yields
\[
{|v_1(t)|} \ge \left( \frac {x_2^\alpha} {\alpha(1-4\del) t + x_2^\alpha}   \right)^{C''/\alpha(1-4\del)}  {|v_1(0)|}
\]
on $(0,t_0]$.  Since again $\alpha(1-4\del) t + x_2^\alpha\le 1$, and also $x_2\ge cv_1(0)$ for some $c=c(\alpha)>0$ when $|v_2(0)|\le  \frac {4\alpha}{1-\alpha} |v_1(0)|$ (because $x,y\in S_\del$ and $\frac 1\del> \frac{4\alpha}{1-\alpha}$), it follows that
\[
{|v_1(t)|} \ge c''  {|v_1(0)|^{(C''+1)/\alpha(1-4\del)}}.
\]
on $(0,t_0]$, with some $c''=c''(\alpha)$.  Since we also have $|v_2(\cdot)|\le  \frac {4\alpha}{1-\alpha} |v_1(\cdot)|$ there,  we obtain
\beq \lb{3.9a}
{|\Phi(y,t)-\Phi(x,t)|} \ge \gamma  {|x-y|^{1/\gamma}}.
\eeq
on $(0,t_0]$ (and hence as long as the trajectories remain in $A_\del$, by the previous argument) whenever $x,y\in S_\del$, with some $\gamma=\gamma(\alpha)$>0.

{\bf Case 2.}
Next we assume that $x,y\in S'_\del:=A_\del\cap\{x_2\ge \del x_1\}$.  We now want to pick small enough $\eps>0$ (depending on $\alpha$ and $\del$, and hence on $\alpha$) such that for any streamline $\{\tilde\psi_\alpha=r\}$ that intersects $A_\del$ (in particular, $r\le \del^{2-\alpha}$), the time it takes for a particle moving on this streamline (with velocity $u$) to traverse the portion lying inside $S'_\del\setminus S_\del$ is shorter than the time it takes for a particle moving on the streamline $\{\tilde\psi_\alpha=r'\}$ for any $r'\in[\frac r2,2r]$ to traverse the portion lying inside $\{\eps\del x_2\le x_1\le \del x_2\}$.  Such $\eps$ exists because for the stream function $\psi(x)=x_1x_2(x_1+x_2)^{-\alpha}$ the ratio of these times only depends on $\frac {r'}r$ (due to $(2-\alpha)$-homogeneity of $\psi$) and converges to 0 as $\eps\to 0$ (uniformly in $\frac {r'}r\in[\frac 12,2]$), and the ratio of any derivative of $\tilde \psi_\alpha$ and the same derivative of $\psi$ is within $[\frac 12,2]$ on all of $A_\del$ if $\del$ is small enough.  The reader may want to keep in mind that since the level sets of $\psi$ are just scaled copies of each other, the same is true asymptotically for the level sets of $\tilde\psi_\alpha$ near the origin.

The case $x,y\in S_\del$ was already covered.  Let us now assume $x,y\in S'_\del\setminus S_\del$, as well as $\tilde\psi_\alpha (x)\le \tilde\psi_\alpha (y)\le 2 \tilde\psi_\alpha (x)$.  Let $t_0\ge 0$ be the first time when both trajectories $\Phi(x,\cdot)$ and $\Phi(y,\cdot)$ have left $S'_\del\setminus S_\del$.  Because of our choice of $\eps$, we must have $\Phi(x,t_0),\Phi(y,t_0)\in \{\eps\del x_2\le x_1\le \del x_2\}$, so \eqref{3.9} shows that while both trajectories stay in $A_\del$, we have
\beq \lb{3.10}
|\Phi(x,t)-\Phi(y,t)|   \le  4\del^{-1}|\Phi(x,t_0)-\Phi(y,t_0)|^\del
\eeq
for $t\ge t_0$. To estimate the  last term, we notice that while either trajectory is within $S'_\del\setminus S_\del$, its velocity is bounded below by $c|x|^{1-\alpha}$ for some $c=c(\alpha)$, which means that $t_0\le C_1|x|^\alpha$  for some $C_1=C_1(\alpha)$.  Since we also have the bound $|\nabla u| \le C_2|x|^{-\alpha}$ with $C_2=C_2(\alpha)$ on the portions of both trajectories lying in $S_{\del'} \cap \{x_1\ge \eps\del x_2 \}$ (recall that $\delta,\eps$ depend only on $\alpha$), a simple integration shows that
\beq \lb{3.11}
|\Phi(x,t)-\Phi(y,t)|\le e^{C_2|x|^{-\alpha}C_1|x|^\alpha}|x-y| = e^{C_2C_1}|x-y|
\eeq
for $t\in[0,t_0]$.   It follows that
\beq \lb{3.12}
|\Phi(x,t)-\Phi(y,t)|   \le  C|x-y|^\del
\eeq
while both trajectories stay in $A_\del$, with $C=C(\alpha)$.

Next assume $x,y\in S'_\del\setminus S_\del$ (then without loss $\tilde\psi_\alpha (x)\le  \tilde\psi_\alpha (y)$) and also that $\tilde\psi_\alpha (x)< 2 \tilde\psi_\alpha (y)$.  This last fact shows that the distance of any two points on the portions of the level sets of $\tilde\psi_\alpha$  containing $x$ and $y$ that lie in $S'_\del\setminus S_\del$ is at most $C_3|x-y|$ for some $C_3=C_3(\alpha)$.  If now $t_0$ is the first time when one of the trajectories exits $S'_\del\setminus S_\del$, we will have
\[
|\Phi(x,t_0)-\Phi(y,t_0)|\le C_3|x-y|,
\]
which means we can finish this case by using the case in the next paragraph.

If now $x\in S'_\del\setminus S_\del$ and $y\in S_\del$, we let $z$ be the point with $\tilde\psi_\alpha(z)=\tilde\psi_\alpha(x)$ and $z_1=\del z_2$.  Let also $t_0$ be the first time when $\Phi(x,\cdot)$ leaves $S'_\del\setminus S_\del$.  Then the fact that the level sets of $\tilde\psi_\alpha$ intersect the line $\{x_1=\del x_2\}$ transversally (with angles uniformly bounded away from 0) shows that  $\max\{|x-z|,|y-z|\}\le C_4|x-y|$ for some $C_4=C_4(\alpha)$.  If now $t_0$ is the first time when $\Phi(x,\cdot)$ leaves $S'_\del\setminus S_\del$, then the argument leading to \eqref{3.11} shows that
\[
|\Phi(t_0,x)-\Phi(t_0,z)|\le e^{C_2C_1}C_4|x-y|,
\]
while \eqref{3.9} shows
\[
|\Phi(t_0,y)-\Phi(t_0,z)|\le 4\delta^{-1}|y-z|^\del.
\]
Adding these to estimate $|\Phi(t_0,x)-\Phi(t_0,y)|$ and then applying \eqref{3.9} on the interval $[t_0,t]$ yields
\beq \lb{3.13}
|\Phi(t,x)-\Phi(t,y)|\le C|x-y|^{\delta^2}
\eeq
while both trajectories stay in $A_\del$, with a new $C=C(\alpha)$.  Since this bound can absorb the bounds in \eqref{3.9} and \eqref{3.12} by adjusting $C$, this finishes Case 2.

{\bf Case 3.}
It remains to consider the case $x\in A_\del\setminus S'_\del$.  If  $y\in S'_\del$, let $z$ be the point with $\tilde\psi_\alpha(z)=\tilde\psi_\alpha(x)$ and $z_2=\del z_1$, and let $t_0$ be the first time when $\Phi(x,\cdot)$ enters $S_\del$ (and hence $\Phi(x,t_0)=(z_2,z_1)$ due to symmetry of $\tilde\psi_\alpha$ in its arguments).  Just as in the last paragraph, it follows that $\max\{|x-z|,|y-z|\}\le C_4|x-y|$, and therefore \eqref{3.13} yields
\[
|\Phi(t_0,y)-\Phi(t_0,z)|\le C|x-y|^{\delta^2},
\]
with a new $C=C(\alpha)$.  On the other hand, the same symmetry  shows that $\Phi(z,t_0)=(x_2,x_1)$, so 
\[
|\Phi(t_0,x)-\Phi(t_0,z)|=|x-z| \le C_4|x-y|.
\]
Again, dding these to estimate $|\Phi(t_0,x)-\Phi(t_0,y)|$ and then applying \eqref{3.9} on the interval $[t_0,t]$ yields
\beq \lb{3.14}
|\Phi(t,x)-\Phi(t,y)|\le C|x-y|^{\delta^3}
\eeq
while both trajectories stay in $A_\del$, with a new $C=C(\alpha)$. 

Finally, if both $x,y\in A_\del\setminus S'_\del$, let $t_0$ be the first time when one of $\Phi(x,\cdot)$ and $\Phi(y,\cdot)$ enters $S'_\del$ (assume it is $\Phi(y,t_0)$).  We can now apply the argument leading to \eqref{3.9a} on the time interval $[t_0,0]$, with time running in reverse and with the roles of the two coordinates reversed.  This yields  
\[
{|x-y|} \ge \gamma  {|\Phi(y,t_0)-\Phi(x,t_0)|^{1/\gamma}},
\]
which together with \eqref{3.14} applied on the time interval $[t_0,t]$ implies
\[
|\Phi(t,x)-\Phi(t,y)|\le C|x-y|^{\delta^3\gamma}
\]
while both trajectories stay in $A_\del$, with a new $C=C(\alpha)$.  This finishes Case 3, and so also the proof for $\Phi$.

{\bf Back to $\Phi'$.}
We now note that $\Phi'(x,t)=\Phi(x,P_\alpha(x)t)$, where $P_\alpha(x):=T_{\psi_\alpha}(\psi_\alpha(x))$, and so 
\[
|\Phi'(x,t)-\Phi'(y,t)| \le |\Phi(x,P_\alpha(x)t)-\Phi(y,P_\alpha(x)t)| + |\Phi(y,P_\alpha(x)t)-\Phi(y,P_\alpha(y)t)|
\]
Boundedness of $P_\alpha$ and the result for $u'$ now show that the first term on the right-hand side is bounded by multiple of some power of $|x-y|$, uniformly in $t\in[0,1]$.  The same is true for the second term because $u'$ is bounded and $P_\alpha$ is H\" older continuous due to boundedness of $\nabla\psi_\alpha$, Lemma \ref{L.3.1}(iii), and the fact that $\psi_\alpha$ has a quadratic critical point at the origin.
\end{proof}

\section{Appendix: Fractional Derivatives and Concavity of $\log \psi_\alpha$}


 
As above, we will again use the notation $x=(x_1,x_2)\in \bbR^2$ in the next two results.
 
\begin{lemma}\label{fractional}
Let $P_2:=(-1,1)^2$ and $\Gamma:=[\{0\}\times(-1,1)]\cup[(-1,1)\times\{0\}]$.  If $f\in C^1_{\rm loc}(P_2\setminus\Gamma)$ and there are $K\geq 0$ and $\beta\in[0,1)$ such that 
\[
|f(x_1,x_2)|\leq K|x_1x_2|^{-\beta},\]
\[|\nabla f(x_1,x_2)|\leq K|x_1x_2|^{-\beta-1}
\]  
for all $(x_1,x_2)\in P_2$, then for any $\gamma\in(0,1)$ and $\epsilon>0$ there is $C_{\gamma,\beta,\epsilon}>0$ such that 
\[|\Lambda^\gamma f(x_1,x_2)|\leq C_{\gamma,\beta,\epsilon}K |x_1x_2|^{-\beta-\gamma-\epsilon}
\]  for all $(x_1,x_2)\in P_2$.
\end{lemma}

This lemma can be trivially extended to the case where $f$ is smooth away from finitely many lines, with a controlled blow-up near those lines. In particular, we have the following corollary.

\begin{corollary}
\label{fractionalsines}
If $f\in C^1_{\rm loc}(P_2\setminus\Gamma)$ and there are $K\geq 0$ and $\beta\in[0,1)$ such that 
\[
|f(x_1,x_2)|\leq K|\sin(\pi x_1)\sin(\pi x_2)|^{-\beta},\]
\[|\nabla f(x_1,x_2)|\leq K|\sin(\pi x_1)\sin(\pi x_2)|^{-\beta-1}
\]  
for all $(x_1,x_2)\in P_2$, then for any $\gamma\in(0,1)$ and $\epsilon>0$ there is $C_{\gamma,\beta,\epsilon}>0$ such that
\[|\Lambda^\gamma f(x_1,x_2)|\leq C_{\gamma,\beta,\epsilon}K |\sin(\pi x_1)\sin(\pi x_2)|^{-\beta-\gamma-\epsilon}
\]  
for all $(x_1,x_2)\in P_2$.
\end{corollary}

\begin{proof}[Proof of Lemma \ref{fractional}]
Without loss of generality, assume that $0<x_1\leq x_2$. Then (with all integrals below having $y\in P_2$), 
\[\Lambda^\gamma f(x)=\int_{|x-y|<\frac{1}{2}x_1}\frac{f(x)-f(y)}{|x-y|^{2+\gamma}}dy+\int_{|x-y|\geq \frac{1}{2}x_1}\frac{f(x)-f(y)}{|x-y|^{2+\gamma}}dy.\]
To estimate the first integral, note that all $y=(y_1,y_2)$ satisfying $|x-y|<\frac{1}{2} x_1$ also satisfy $y_1\geq \frac{1}{2} x_1$ and $y_2\geq \frac{1}{2} x_2$. 
Thus, if $|x-y|<\frac{1}{2}x_1,$ we have
\[|f(x)-f(y)|=|f(x)-f(y)|^{1-\gamma-\epsilon}|f(x)-f(y)|^{\gamma+\epsilon}\leq 2K (x_1x_2)^{-\beta-\gamma-\epsilon}|x-y|^{\gamma+\epsilon},\]
where in the last step we used the two assumed inequalities to estimate the two powers of $|f(x)-f(y)|$.
We then get
\[\int_{|x-y|<\frac{1}{2}x_1}\frac{|f(x)-f(y)|}{|x-y|^{2+\gamma}}dy \leq 2K(x_1x_2)^{-\beta-\gamma-\epsilon}\int_0^{x_1} r^{-1+\epsilon}dr \le \frac{2K}{\epsilon}(x_1x_2)^{-\beta-\gamma-\frac{\epsilon}{2}}.\] 

It therefore suffices to consider the second integral. 
We split it into
\[
\int_{\frac{1}{2}x_1 \le |x-y| <\frac{1}{2}x_2} \frac{|f(x)-f(y)|}{|x-y|^{2+\gamma}} dy+\int_{|x-y|\geq\frac{1}{2}x_2}\frac{|f(x)-f(y)|}{|x-y|^{2+\gamma}} dy=I+II.
\]
First we estimate $II$ via
\[
|II|\leq C_\gamma K(x_1x_2)^{-\beta}\int_{\frac 12 x_2}^3 r^{-1-\gamma}dr+K\int_{|x-y|\geq \frac{1}{2}x_2}\frac{|y_1y_2|^{-\beta}}{|x-y|^{2+\gamma}}dy.
\]
The first term is no more than $C_\gamma K(x_1x_2)^{-\beta}x_2^{-\gamma}\le C_\gamma K(x_1x_2)^{-\beta-\gamma}$ (with a new constant) and for the second we use H\"older's inequality to obtain 
\begin{align*}
\int_{|x-y|\geq \frac{1}{2}x_2}\frac{|y_1y_2|^{-\beta}}{|x-y|^{2+\gamma}}dy  
& \leq \left(\int_{P_2} |y_1y_2|^{-(1-\delta)} dy \right)^{\frac{\beta}{1-\delta}} \left(\int_{|x-y|\geq \frac{1}{2}x_2} |x-y|^{-\frac{2+\gamma}{1-\frac{\beta}{1-\delta}}} dy \right)^{1-\frac{\beta}{1-\delta}}
\\ & \leq C_\delta x_2^{-(2+\gamma)+2(1-\frac{\beta}{1-\delta})}=C_\delta x_2^{-\gamma-\frac{2\beta}{1-\delta}},
\end{align*}
for any $\delta\in(0,1-\beta).$ We choose such $\delta$ so that $\frac{2\beta}{1-\delta}<2\beta+\gamma$, and then $x_1\leq x_2$ yields 
\[|II|\leq C_{\gamma,\beta}K(x_1x_2)^{-\beta-\gamma}.\]
We similarly estimate
\[
|I|\leq C_\gamma K(x_1x_2)^{-\beta}\int_{\frac 12 x_1}^{\frac 12 x_2} r^{-1-\gamma}dr+ C_\beta K x_2 ^{-\beta} \int_{\frac{1}{2}x_1 \le |x-y| <\frac{1}{2}x_2} \frac{|y_1|^{-\beta}}{|x-y|^{2+\gamma}}dy.
\] 
The first term is no more than $C_\gamma K(x_1x_2)^{-\beta}x_1^{-\gamma}\le C_\gamma K(x_1x_2)^{-\beta-\gamma}$ (with a new constant) and the second is no more than
\[
C_{\beta}K x_2^{-\beta}\int_{|y_1|<\frac{1}{2}x_1}\frac{|y_1|^{-\beta}}{|x-y|^{2+\gamma}}dy+C_{\beta,\gamma} K(x_{1}x_2)^{-\beta} \int_{\frac 12 x_1}^{\frac 12 x_2} r^{-1-\gamma}dr.
\] 
The second of these terms is again estimated by $C_{\beta,\gamma} K(x_1x_2)^{-\beta}x_1^{-\gamma}\le C_{\beta,\gamma} K(x_1x_2)^{-\beta-\gamma}$, and the first by 
\[
C_{\beta,\gamma}K x_2^{-\beta}\int_{|y_1|<\frac{1}{2}x_1}\frac{|y_1|^{-\beta}}{(x_1^2+(x_2-y_2)^2)^{1+\frac{1}{2}\gamma}}dy\le
 C_{\beta,\gamma}K x_2^{-\beta}x_1^{1-\beta}\int_{0}^\infty\frac{dz}{(x_1^2+z^2)^{1+\frac{1}{2}\gamma}}.
 \]
This is no more than $C_{\beta,\gamma} Kx_2^{-\beta}x_1^{1-\beta} x_1^{-1-\gamma}\leq C_{\beta,\gamma}K(x_1x_2)^{-\beta-\gamma}$, concluding the proof. 
\end{proof}

\begin{lemma} \label{L.3.7}
If $\psi_\alpha$ is the function from Section \ref{S3},
then $\log\psi_\alpha$ is concave for any $\alpha\in [0,1]$. 
\end{lemma}

\begin{proof}
To simplify notation, let us instead consider the function $f_\alpha:=\log(2^{-\alpha}\psi_\alpha(\frac{\cdot}{\pi})),$ so that
\[f_\alpha(x,y)=\log(\sin(x))+\log(\sin(y))-\alpha\log( \sin(x)+\sin(y)).\]
We have
\begin{align*}
\partial_{xx}\log(\sin(x)) & 
=-\frac{1}{\sin^2(x)},
\\ \partial_{xx} \log(\sin(x)+\sin(y)) & 
=-\frac{1+\sin(x)\sin(y)}{(\sin(x)+\sin(y))^2}
\\ \partial_{yy}\log(\sin(x)+\sin(y)) & =-\frac{1+\sin(x)\sin(y)}{(\sin(x)+\sin(y))^2}
\\ \partial_{xy}\log(\sin(x)+\sin(y)) & =-\frac{\cos(x)\cos(y)}{(\sin(x)+\sin(y))^2}
\end{align*}
Now with $a:=\sin(x)$ and $b:=\sin(y)$ we obtain
\[
\partial_{xx} f_\alpha=-\frac{1}{a^2}+\alpha\frac{1+ab}{(a+b)^2}\leq-\frac{1}{a^2}+\frac{1+ab}{(a+b)^2} \le -\frac{b^2}{(a+b)^2a^2}\leq 0
\]
because $a,b,\alpha\in[0,1]$. Symmetry yields $\partial_{yy}f_\alpha\le -\frac{a^2}{(a+b)^2b^2}\leq 0$, so
 \[\tr D^2 f_\alpha=\Delta f_\alpha \leq 0.\] 
We also have
\[
\det D^2 f_\alpha=\partial_{xx}f_\alpha\partial_{yy}f_\alpha- (\partial_{xy}f_\alpha)^2 \ge \frac{b^2}{(a+b)^2a^2} \frac{a^2}{(a+b)^2b^2} - \frac{(1-a^2)(1-b^2)}{(a+b)^4} \ge 0,
\]
and concavity of $f_\alpha$ follows.
\end{proof}


\begin{thebibliography}{99}
\bibitem{ACM} G. Alberti, G. Crippa, and A.L. Mazzucato. Exponential self-similar mixing and loss of regularity for continuity equations. 
\emph{C. R. Acad. Sci. Paris, Ser. I,}  352:901--906, 2014.

\bibitem{ACM2} G. Alberti, G. Crippa, and A.L. Mazzucato. 
Exponential self-similar mixing by incompressible flows. \emph{J. Amer. Math. Soc.,} to appear.

\bibitem{ACM3} G. Alberti, G. Crippa, and A.L. Mazzucato. 
Loss of regularity for the continuity equation with non-Lipschitz velocity field. \emph{Preprint,} arXiv:1802.02081.

\bibitem{Anosov} D. V. Anosov. Geodesic flows on closed Riemannian manifolds with negative curvature. 
\emph{Proc. Steklov Inst.}  {\bf 90}, 1967.



\bibitem{BBS} J. Bedrossian, A. Blumenthal, and S. Punshon-Smith. Lagrangian chaos and scalar advection in stochastic fluid mechanics. \emph{Preprint,} arXiv:1809.06484.

\bibitem{BOS} Y. Brenier, F. Otto, and C. Seis. Upper bounds on the coarsening rates in demixing binary viscous fluids. \emph{SIAM J. Math. Anal.}, 43:114--134, 2011.

\bibitem{Bressan} A. Bressan.
 A lemma and a conjecture on the cost of rearrangements.
\emph{Rend. Sem. Mat. Univ. Padova}, 110:97--102, 2003. 

\bibitem{B2} A. Bressan. Prize offered for the solution of a problem on mixing flows. http://www.math.psu.edu/bressan/PSPDF/prize1.pdf, 2006.

\bibitem{CKRZ} P. Constantin, A. Kiselev, L. Ryzhik, and A. Zlato\v{s}. Diffusion and mixing in fluid flow. \emph{Ann.
of Math. (2)}, 168(2):643--674, 2008.

\bibitem{CFS} I.P. Cornfeld, S.V. Fomin, and Y.G. Sinai. \emph{Ergodic Theory}. {Springer-Verlag,} New York, 1982.

\bibitem{CZDE} M. Coti-Zelati, M.G. Delgadino, and T.M. Elgindi. On the relation between enhanced dissipation time-scales and mixing rates. \emph{To appear in Comm. Pure Appl. Math.} arXiv:1806.03258.

\bibitem{CL} G. Crippa and C. De Lellis. Estimates and regularity results for the DiPerna-Lions flow. \emph{J. Reine Angew. Math.}, 616:15--46, 2008.

\bibitem{CLS} G. Crippa, R. Luc\'{a}, and C. Schulze. Polynomial mixing under a certain stationary Euler flow. \emph{Preprint,} arXiv:1707.09909.

\bibitem{CS} G. Crippa and C. Schulze. Cellular mixing with bounded palenstrophy. \emph{Preprint,} arXiv:1707.01352.

\bibitem{DePauw} 
N. Depauw.
Non unicit\' e des solutions born\' ees pour un champ de vecteurs BV en dehors d'un hyperplan.  \emph{C. R. Math. Acad. Sci. Paris}, 337(4):249--252, 2003.

\bibitem{DT} C. R. Doering and J.-L. Thiffeault. Multiscale mixing efficiencies for steady sources. \emph{Phys. Rev. E}, 74 (2), 025301(R), August 2006. 

\bibitem{Dolgopyat} D. Dolgopyat. On decay of correlations in Anosov flows. 
\emph{Ann. of Math.}  {\bf 147} (1998), 357-390.

\bibitem{Dolgopyat2} D. Dolgopyat. Personal communication. 

\bibitem{DKK} D. Dolgopyat, V. Kaloshin, and L. Koralov. Sample path properties of the stochastic flows. \emph{Ann. of Prob.} {\bf 32} (2004), 1-27. 

\bibitem{IF} Y. Feng and G. Iyer. Dissipation Enhancement by Mixing. \emph{Preprint,} arXiv:1806.03699.

\bibitem{IKX} G. Iyer, A. Kiselev, and X. Xu. Lower bounds on the mix norm of passive scalars advected by incompressible enstrophy-constrained flows. \emph{Nonlinearity}, 27(5):973--985, 2014.

\bibitem{K} A. Katok. Bernoulli diffeomorphisms on surfaces. \emph{Ann. of Math.} {\bf 110} (1979), 529-547. 

\bibitem{LTD} Z. Lin, J. L. Thiffeault, and C. R. Doering. Optimal stirring strategies for passive scalar mixing. \emph{J. Fluid Mech.}, 675:465--476, 2011.

\bibitem{Liverani} C. Liverani. On contact Anosov flows. 
\emph{Ann. of Math.}  {\bf 159} (2004), 1275-1312.

\bibitem{LLNMD} E. Lunasin, Z. Lin, A. Novikov, A. Mazzucato, and C. R. Doering. Optimal mixing and optimal stirring for fixed energy, fixed power, or fixed palenstrophy flows. \emph{J. Math. Phys.}, 53(11):115611, 15, 2012.

\bibitem{MMP} G. Mathew, I. Mezi\'c, and L. Petzold. A multiscale measure for mixing. \emph{Physica D}, 211(1-2):23--46, 2005.

\bibitem{M} V. Maz'ya, \emph{Sobolev Spaces. With Applications to Elliptic Partial Differential Equations}, 2nd, revised and augmented ed., Springer, Berlin, 2011.

\bibitem{NPV} E. Di Nezza, G. Palatucci, and E. Valdinoci. Hitchhiker's guide to the fractional Sobolev spaces, \emph{Bull. Sci. math.}, 136(5):521--573, 2012.

\bibitem{OSS} F. Otto, C. Seis, and D. Slep\v{c}ev. Crossover of the coarsening rates in demixing of binary viscous liquids. \emph{Commun. Math. Sci.} 11:441--464, 2013.

\bibitem{Pie} R. Pierrehumbert, \it Tracer microstructure in the large-eddy dominated regime, \rm Chaos, Solitons \& Fractals {\bf 4} (1994), 1091--1110.

\bibitem{S} C. Seis. Maximal mixing by incompressible fluid flows. \emph{Nonlinearity}, 26(12):3279--3289, 2013.

\bibitem{S2} D. Slep\v{c}ev. Coarsening in nonlocal interfacial systems. \emph{SIAM J. Math. Anal.}, 40(3):1029--1048, 2008.

\bibitem{STD} T. A. Shaw, J.-L. Thiffeault, and C. R. Doering. Stirring up trouble: multi-scale mixing measures for steady scalar sources. \emph{Phys. D}, 231(2):143--164, 2007.

\bibitem{SpringhamThesis} J. Springham. Ergodic properties of linked-twist maps. \emph{PhD Thesis,} University of Bristol, 2008.

\bibitem{T} J.-L. Thiffeault. Using multiscale norms to quantify mixing and transport. \emph{Nonlinearity}, 25(2):R1--R44, 2012.

\bibitem{TDG} J.-L. Thiffeault, C. R. Doering, and J. D. Gibbon. A bound on mixing efficiency for the advection-diffusion equation. \emph{J. Fluid Mech.}, 521:105--114, 2004.

\bibitem{YaoZla} Y. Yao and A. Zlato\v{s}. Mixing and un-mixing by incompressible flows. \emph{J. Eur. Math. Soc.}, 19(7): 1911-1948, 2017. 

\bibitem{Zill} C. Zillinger. On geometric and analytic mixing scales: comparability and convergence rates for transport problems. \emph{Preprint,} arXiv:1804.11299.

\end{thebibliography}
\end{document}